\documentclass[11pt]{amsart}%
\usepackage{geometry}
\usepackage{amssymb, comment,wasysym}
\usepackage{mathrsfs}
\usepackage{mathscinet}
\usepackage{hyperref}
\usepackage{amsmath}
\usepackage{amsfonts}
\usepackage{amssymb}
\usepackage{graphicx, enumerate, caption, subcaption}
\usepackage{color}
\providecommand{\U}[1]{\protect\rule{.1in}{.1in}}
\newtheorem{thm}{Theorem}
\newtheorem{lem}{Lemma}

\newtheorem{prop}{Proposition}

\newtheorem{algorithm}{Algorithm}
\theoremstyle{remark}

\begin{document}
\title[Perfect Sampling of Generalized Jackson Networks]{Perfect Sampling of Generalized Jackson Networks}
\author{J. Blanchet}
\address{Columbia University \\
Columbia University, Department of Industrial Engineering \& Operations Research.}
\email{jose.blanchet@columbia.edu}
\author{X. Chen }
\address{Wuhan University\\
Wuhan University, Economics and Management School}
\email{little4cat@gmail.com}
\date{\today}

\begin{abstract}
We provide the first perfect sampling algorithm for a Generalized Jackson
Network of FIFO queues under arbitrary topology and non-Markovian assumptions
on the input of the network. We assume (in addition to stability) that the
interarrival and service times of customers have finite moment generating
function in a neighborhood of the origin, and the interarrival times have
unbounded support.

\end{abstract}
\maketitle

%\footnote{Corresponding author: J. Blanchet. 500 W 125th St. New York, NY.
%10027. email: jose.blanchet@columbia.edu}

%\bibliographystyle{plainnat}

\section{Introduction}

We present the first perfect sampling algorithm (i.e. unbiased sampling also
known as exact simulation) for the steady-state of so-called Generalized
Jackson Networks (GJNs).

A precise description of a GJN consists of $d$ single server queueing
stations, with infinite capacity waiting rooms and each operating under a
standard FIFO protocol. The $i$-th station receives arrivals from outside the
network (i.e. external arrivals) according to a renewal process with arrival
rate $\lambda_{i}\in\lbrack0,\infty)$ (note that $\lambda_{i}=0$ is possible,
meaning that the $i$-th station does not receive external arrivals, but we
assume that $\lambda_{i}>0$ for some $i\in\{1,...,d\}$). All the renewal
arrival processes are independent. We use $\lambda=\left(  \lambda
_{1},...,\lambda_{d}\right)  ^{T}$ to denote the vector of arrival rates.
(Throughout this paper all vectors are column vectors unless otherwise stated,
and we use $^{T}$ to denote transposition.)

All the service requirements are independent. Inter-arrival times and service
requirements are all independent. The mean service time at station $i$ is
$1/\mu_{i}\in\left(  0,\infty\right)  $. We use $\mu=\left(  \mu_{1}%
,...,\mu_{d}\right)  ^{T}$ to denote the vector of service rates. The service
requirements at station $i$ are i.i.d. (independent and identically distributed).

Immediately after a customer is served at the $i$-th station, he will go to
station $j$ with probability $Q_{i,j}\in\lbrack0,1]$ for $j\in\{1,...,d\}$ and
he will leave the network with probability $Q_{i,0}=1-\sum_{j=1}^{d}Q_{i,j}$.
We write $Q=(Q_{i,j}:1\leq i,j\leq d)$ for the associated $d\times d$
substochastic routing matrix. The network is assumed to be open in the
sense that $Q^{n}\rightarrow0$ as $n\rightarrow\infty$. We
assume, without the loss of generality, that $Q_{i,i}=0$. Otherwise we can redefine the service requirements
via a geometric convolution with success probability equal to $1-Q_{i.i}$ and
thus represent the network in terms of a model in which $Q_{i,i}=0$.

The so-called flow equations are given by
\begin{equation}
\phi_{i}=\lambda_{i}+\sum_{j=1}^{d}Q_{j,i}\phi_{j}, \label{EQ_FLOW_INTRO}%
\end{equation}
which implies that $\phi=\left(  \phi_{1},...,\phi_{d}\right)  ^{T}$ satisfies
$\phi=\left(  I-Q^{T}\right)  ^{-1}\lambda$. (Note that $\left(  I-Q\right)
^{-1}=I+Q+Q^{2}+....$ is well defined because the network is open.)

Under the previous setup, the GJN is stable (in the sense of possessing a
steady-state distribution for the workload and queue length processes at each
station) if and only if
\begin{equation}
\phi<\mu, \label{Eq:Stability}%
\end{equation}
where the inequality is understood componentwise.

Under mild assumptions (including for example the case of Poisson arrivals or
phase-type inter-arrival and service times) we provide the first exact
simulation algorithm for a Generalized Jackson Network. (The precise
assumptions, listed as Assumptions 1-4, are given in Section
\ref{Section_Description_GJN}.) All previous algorithms operate under
more-restrictive assumptions relative to what is required in our algorithm.
The more restrictive assumptions include: a) The networks are Markovian (i.e.
inter-arrivals and service times are assumed to be exponential), or b) The
networks are bounded (i.e. the stations are assumed to have rooms with finite
buffer sizes); see, for example, \cite{Busi15} and \cite{MurdochTakahara2006}.

The work of \cite{Blanchet&Chen} is closest in spirit to our algorithm here.
The authors in \cite{Blanchet&Chen} consider a so-called stochastic fluid
network (SFN), which is much simpler than a GJN because there is much less
randomness in the system. Customers that arrive at station $i$ in a SFN bring
service requirements which are i.i.d., this part is common to the GJN model.
However, the workload is processed and transmitted to the stations in the
network in the form of a fluid; so $Q_{i,j}$ represents the exact proportion of flow from
station $i$ to $j$. Therefore, in particular, in a SFN there is no concept of
queue-length. In addition, the SFNs treated in \cite{Blanchet&Chen} has
Poisson or Markov modulated arrivals and so even the arrival processes that we
consider here are more general. We extend the algorithm in
\cite{Blanchet&Chen} in order to deal with arbitrary renewal processes (as
opposed to only Poisson arrivals), the condition on Assumption 2 is needed to
apply the technique of \cite{Blanchet&Chen} based on a suitable exponential
tilting (see also \cite{Ensor&Glynn} and \cite{Blanchet&Sigman}), this
connection to exponential changes of measure explains the need for Assumption 3.

The algorithm in \cite{Blanchet&Chen} allows to obtain a sample from the
maximum from time 0 to infinity, of a multidimensional random walk with negative drift. Here we extend
the algorithm to sample from the running maximum (componentwise), that is, the maximum from time
$n$ to infinity, for all $n\geq0$. Our extension is given in Algorithm
\ref{algo last}.

The real difficulty in doing perfect sampling of GJNs, however, arises from
the fact that each customer might bring an arbitrarily long sequence of
service requirements, because the description of the routing topology admits
the possibility of visiting a given station multiple times. In addition,
contrary to SFN's, GJN's are not monotone in their initial condition. This
lack of monotonicity introduces challenges when applying standard perfect
simulation techniques.

Our strategy is to apply Dominated Coupling From The Past (DCFTP), which
requires the use of a suitable dominating process simulated backwards in time
and in stationarity. We are able to use sample path comparison results
developed by \cite{Chang&etal}, which allow us to bound the total number of
customers in the GJN by a set of suitably defined autonomous queues which are
correlated. In addition, we provide additional sample path comparison results
which are of independent interest (see Theorem \ref{Thm_Domination}).

We need to simulate, backwards in time, stationary and correlated autonomous
queues. These processes can be represented, componentwise, in terms of an
infinite horizon maximum of the difference of superposition of renewal
processes (the difference having negative drift so the infinite horizon
maximum is well defined). The fact that the queues are correlated comes from
the fact that each jump in the renewal processes may correspond to a departure from one station, and at the same time, an arrival to another station due to the internal routing. We are able to extend the technique in
\cite{Blanchet&Chen} in order to deal with multidimensional and correlated
renewal processes and thus complete the application of the DCFTP protocol.

The rest of the paper is organized as follows. In Section 2, we briefly
discussing how DCFTP operates and describe the GJN. In Section 3, we construct a class of
dominating processes which will be useful for our development. We provide a
general overview of our algorithm and the main result of the paper in Section 4. Then we proceed by describing how to implement the subroutines
of our algorithm in Section 5 and finish the paper with a numerical experiment
in Section 6.

\section{An introduction to DCFP and GJN}
\subsection{Elements of Dominated Coupling From The Past}

Let us first provide a general description of DCFTP. Consider a stationary
process $\left(  Y\left(  t\right)  :t\in\left(  -\infty,\infty\right)
\right)  $, we are interested in sampling from $Y\left(  0\right)  $. Suppose
that the following is available to the simulator:

\begin{itemize}
\item[\textit{DCFTP 1}] A pair of stochastic processes $\left(  Y^{-}\left(
t\right)  :t\in\left(  -\infty,\infty\right)  \right)  $ and $\left(
Y^{+}\left(  t\right)  :t\in\left(  -\infty,\infty\right)  \right)  $ coupled
in such a way that $Y^{-}\left(  t\right)  \preccurlyeq Y\left(  t\right)
\preccurlyeq Y^{+}\left(  t\right)  $ for all $t$, where \textquotedblleft%
$\preccurlyeq$\textquotedblright\ is any partial order.

\item[\textit{DCFTP 2}] It is possible to simulate $\bar{\omega}:=\left(
Y^{-}\left(  t\right)  ,Y^{+}\left(  t\right)  :t\in\lbrack-T,0]\right)  $ for
a (finite almost surely) time $-T$ in the past such that: a)$\ Y^{+}\left(
-T\right)  =Y^{-}\left(  -T\right)  $, and b) $Y\left(  0\right)  $ can be
obtained from the information used to generate $\bar{\omega}$.
\end{itemize}

A time $-T$ satisfying the conditions in \textit{DCFTP 2} is known as a
\textit{coalescence time}.

Generally, at least in the setting of Markov processes, the condition that
$Y^{+}\left(  -T\right)  =Y^{-}\left(  -T\right)  $ combined with
\textit{DCFTP 1} above indicates that the value of $Y\left(  -T\right)  $ is known
and therefore at least the marginal\ evolution of $Y\left(  \cdot\right)  $ is
completely determined, and so is the value of $Y(0)$. However, it is important to keep in mind that the
processes $Y^{+},Y^{-}$ and $Y$ must remain coupled.

The validity of DCFTP is proved in \cite{Kendall}; the method is an extension
of CFTP, which was proposed in the seminal paper of \cite{Propp&Wilson}.
Intuitively, the idea is that if one could simulate the path $\left(
Y^{-}\left(  t\right)  ,Y\left(  t\right)  ,Y^{+}\left(  t\right)
:t\leq0\right)  $, from the infinite past, then one could obtain  $Y\left(
0\right)  $ in stationality. However, since we can simulate
$\bar{\omega}$ in finite time and use this information to reconstruct
$Y\left(  0\right)  $, we do not need to simulate the process from the
infinite past.

Obtaining the elements described in bullets \textit{DCFTP 1-2} above often
requires several auxiliary constructions. In our particular application $Y\left(  \cdot\right)  $ corresponds to the
number in system in each station (so $Y\left(  \cdot\right)  $ is a
$d$-dimensional process) and we shall set $Y^{-}\left(  t\right)  =0$. The partial order relationship \textquotedblleft$\preccurlyeq$%
\textquotedblright\ is based on the sum of the coordinates (i.e $x\preccurlyeq
y$ if and only if $\sum x_{i}\leq\sum y_{i}$).

The process $Y^{+}\left(  \cdot\right)  $ is the one that will require
auxiliary constructions, we shall first construct an auxiliary process $Y^{0}
$ which dominates $Y$ based on artificially increasing (just slightly) the
service requirement of all stations in the GJN. Then we will construct $Y^{+}$
which is a process similar to a GJN, except that the servers will enjoy
vacation periods whenever there is no customer waiting in queue to be served. Finally, we will need an additional process, $Y^{\prime}$, which is corresponding to the autonomous queues and  will allow us
to identify the coalescence time $-T$.

\subsection{Description of the GJN\label{Section_Description_GJN}}

In this section, we give detailed description and assumptions of the
generalized Jackson network (GJN) we are going to simulate.

We consider a GJN consisting of $d$ service stations and each station has a
single server. In the rest of our paper, we shall denote the GJN by
$\mathcal{N}$. The basic assumptions of the GJN $\mathcal{N}$ is as follows: .

\begin{itemize}
\item \textit{Arrival times}: Customers arrive (from the external world) at
station $i$ according to some renewal process with i.i.d. interarrival times
$U_{i}(n)$. In particular, $U_{i}(n)=A_{i}(n)-A_{i}(n-1)$ where $A_{i}\left(
n\right)  $ is the arrival time of the $n$-th customer of station $i$. The
arrival rate $\lambda_{i}$ is defined as $E[U_{i}\left(  k\right)]
=1/\lambda_{i}\in(0,\infty]$. If $\lambda_{i}=0$ then $A_{i}\left(  n\right)
=\infty$. (By convention we let $\lambda_{i}U_{i}\left(  1\right)  =0$ if
$\lambda_{i}=0$.)

\item \textit{Service times}: $\sigma_{i}(k)$ is the service time of the
$k$-th customer that is served in station $i$. $\{\sigma_{i}(k)\}$ is a i.i.d.
sequence and independent of the arrival times, routing indicators and service
times of the other stations. The service rate $\mu_{i}$ is defined as
$E[\sigma_{i}(k)]=1/\mu_{i}$.

\item \textit{Routing mechanism}: After finishing service, the $k$-th customer
in station $i$ is assigned with a routing indicator $r_{i}\left(  k\right)
\in\{0, 1, 2, ..., d\} $ and it will leave the network immediately if
$r_{i}(k)=0$, or join the queue of station $r_{i}(k)$ otherwise. $\{r_{i}(k)\}$
is a i.i.d. sequence and independent of the arrival times, service times and
routing indicators of the other stations. The routing probability $Q_{ij}$ is
defined as $Q_{ij}=P(r_{i}(k)=j)$.
\end{itemize}

Clearly the sequences $\{A_{i}\left(  n\right)  :n\geq1\}$ together with
$\{\left(  r_{i}\left(  k\right)  ,\sigma_{i}\left(  k\right)  \right)
:k\geq1\}$ for $i\in\{1,...,d\}$ are enough to fully describe the evolution of
the queueing network, assuming that the initial state of the network is given. So,
let us assume that the network is initially empty and let us write
$Y_{i}\left(  t\right)  $ to denote the number of customers in the $i$-th
service station at time $t$, including both in the queue and in service, for $i\in\{1,...,d\}$. As noted in the Introduction, the
flow equations are given in equation (\ref{EQ_FLOW_INTRO}), the vector
$\phi_{i}$'s in $\phi=\left(  \phi_{1},...,\phi_{d}\right)  ^{T}$ are called
the net-input rates of the GJN.

In addition to the stability condition given in (\ref{Eq:Stability}),
throughout this paper we shall impose the following assumptions:

\textbf{Assumptions:}

\begin{enumerate}[1.]
\item The inter-arrival times have unbounded support. That is, if $\lambda
_{i}>0$ then $P\left(  U_{i}\left(  1\right)  >m\right)  >0$ for all
$m\in\left(  0,\infty\right)  $.

\item There exists $\delta>0$ such that for all $i$
\begin{align}
\sup_{t\geq0}E[\exp\left(  \delta\lambda_{i}\left(  U_{i}\left(  1\right)
-t\right)  \right)  |U_{i}\left(  1\right)   &  >t]<\infty,\label{Eq:AS_2}\\
\sup_{t\geq0}E[\exp\left(  \delta\left(  \sigma_{i}\left(  1\right)
-t\right)  \right)  |\sigma_{i}\left(  1\right)   &  >t]<\infty.\nonumber
\end{align}
In particular, $\lambda_{i}U_{i}\left(  1\right)  $ and $\sigma_{i}\left(
1\right)  $ have a finite moment generating function for all $i$.

\item The inter-arrival times and service times can be individually simulated
exactly, and moreover, we can simulate from exponentially tiltings (i.e. the
natural exponential family) associated to these distributions -- see equation
(\ref{Eq:Exp_Tilt}).

\item The inter-arrival times and service times have a continuous distribution.
\end{enumerate}

Assumptions 1 to 4 are relatively mild and encompass a large class of models
of interest including Poisson arrivals and phase-type service time
distributions (and mixtures thereof). We shall also discuss immediate
extensions to the case of Markov modulated GJNs. Assumption 1 ensure that the
network will empty infinitely often with probability one. We require the
existence of a finite moment generating function because we will apply an
extension of a technique developed in \cite{Blanchet&Chen}, which is based on
exponential tiltings and importance sampling, therefore the need for
Assumption 3. We need the uniformity on exponential moments for conditional
excess distributions in Assumption 2 because we apply a Lyapunov bound similar
to that developed by \cite{Gamarnik&Zeevi}. However, we believe that this
uniformity requirement is a technical condition and that our main result holds
assuming only that (\ref{Eq:AS_2}) is satisfied for $t=0$. Finally, Assumption
4 is introduced for simplicity to avoid dealing with simultaneous events.

Under Assumptions 1 to 4 we provide an algorithm for sampling from the
steady-state queue-length and workload processes at each station in the
network. The number of random variables required to terminate our proposed
procedure has a finite moment generating function in a neighborhood of the
origin (in particular the expected termination time of the algorithm is finite).

\section{Construction of the Auxiliary and Dominating Processes}

In this section, we shall construct two dominating processes for $Y(t)$, related
to vacation queues and autonomous queues.

\subsection{An Auxiliary GJN}

Before constructing the two bounding systems, we need to construct an
auxiliary upper bound GJN, which we shall denote by $\mathcal{N}^{0}$. The
auxiliary GJN $\mathcal{N}^{0}$, is obtained from the original GJN,
$\mathcal{N}$, by slightly decreasing the service rates at each station while
keeping the network stable. In particular, we shall select constants
$a_{i}\geq1$ for $i\in\{1,...,d\}$ momentarily. We define $\sigma_{i}%
^{0}\left(  k\right)  =\sigma_{i}\left(  k\right)  a_{i}$, and correspondingly
set $\mu_{i}^{0}=\mu_{i}/a_{i}$ for $a_{i}\geq1$ so that
$\mu^{0}=\left(  \mu_{1}^{0},...,\mu_{d}^{0}\right)  ^{T}$, satisfies,%
\begin{equation}
\lambda<\left(  I-Q^{T}\right)  \mu^{0}, \label{Eq:Stability_Bound}%
\end{equation}
componentwise. It is always possible to pick
$a_{i}\geq1$ satisfying (\ref{Eq:Stability_Bound}). In order to see this,
reason as follows. First, define $\mu^{0}=\left(  I-Q^{T}\right)  ^{-1}\left(
\lambda+\delta e\right)  $ (where $e$ is the vector of ones and $\delta>0$ is
to be chosen). Since $\phi=\left(  I-Q^{T}\right)  ^{-1}\lambda<\mu$ and the
matrix $\left(  I-Q^{T}\right)  ^{-1}$ has non-negative elements, we can
choose $\delta>0$ small enough so that $\mu_{i}>\mu_{i}^{0}$ and therefore
$a_{i}=\mu_{i}/\mu_{i}^{0}>1$. Moreover, by definition
\[
\left(  I-Q^{T}\right)  \mu^{0}=\lambda+\delta e>\lambda.
\]

The evolution of $\mathcal{N}^{0}$, initially empty, is also
fully described by the sequences $\{A_{i}\left(  n\right)  :n\geq1\}$ and
$\{\left(  r_{i}\left(  k\right)  ,\sigma_{i}^{0}\left(  k\right)  \right)
:k\geq1\}$, $i\in\{1,...,d\}$, where $\sigma_{i}^{0}(k)=a_{i}\sigma_{i}(k)$.
Let $Y_{i}^{0}\left(  t\right)  $ be the number of customers in the $i$-th
service station at time $t$ (including both in queue and in service), for
$i\in\{1,...,d\}$. As we shall review in Theorem \ref{Thm_Domination}, given
the same initial condition at time 0, $\sum_{i}Y_{i}^{0}(t)\geq\sum_{i}%
Y_{i}(t)$, for all $t$ ; this is intuitive since every customer in
$\mathcal{N}^{0}$ needs more service time at every station than in $\mathcal{N}$.

%\textbf{Remark:} Note that the arrival process has not been modified at all,
%only the service requirements. Also, observe that the routing directions are
%also intact and the way in which the service times are labeled, in the order
%in which they are processed by the servers, has also remained intact. It is
%important to note, however, that the modification on service requirements
%implies that the $n$-th arrival to the network might now be routed differently
%throughout the network (in particular its service requirements might not
%correspond at all to the service requirements assigned in the construction of
%the GJN $\mathcal{N}$).

\subsection{The Vacation System}

We now describe the bonding system consisting of vacation queues, which we
shall denote by $\mathcal{N}^{+}$. The system $\mathcal{N}^{+}$ evolves
following almost the same rules as $\mathcal{N}^{0}$ except that, whenever the
$i$-th server completes a service \textit{and} no customer is waiting in queue
to be served, the server enters a vacation period following the same
distribution of $\sigma^{0}_{i}(k)$. The vacation periods are all independent, and also independent of the arrival times, service times and routing
indicators. If at least one customer is waiting in queue, the server will work
on the service requirement of the first customer waiting in queue.

In more detail, the vacation periods are not interrupted when a new customer
arrives, instead the customer waits until the server finishes its current
activity (current vacation or service). Moreover, if after completing a
vacation the server still finds the queue empty, a new vacation period starts, and the server keeps taking vacation periods until, upon return
of a vacation, the server finds at least one customer present in the queue,
waiting to be served.

The evolution of the vacation system $\mathcal{N}^{+}$, coupled with
$\mathcal{N}^{0}$, is fully described by the sequences $\{A_{i}\left(
n\right)  :n\geq1\}$ $\{\left(  r_{i}\left(  k\right)  ,\sigma_{i}^{0}\left(
k\right)  \right)  :k\geq1\}$, $i\in\{1,...,d\}$, along with the vacation
period sequence $\left\{  v_{i}^{0}\left(  k\right)  :k\geq1\right\}  $. For
each $i$, the sequence $\left\{  \upsilon_{i}^{0}\left(  k\right)
:k\geq1\right\}  $ is an i.i.d. copy of the sequence $\left\{  \sigma_{i}%
^{0}\left(  k\right)  :k\geq1\right\}  $. The random variable $\upsilon
_{i}^{0}\left(  k\right)  $ denotes the $k$-th vacation period taken by the
$i$-th server.

Let us write $Y_{i}^{+}\left(  t\right)  $ to denote the number of customers
in the $i$-th station at time $t$ (including both in queue and in service). As
stated in Theorem \ref{Thm_Domination} below, we have that, given the same
initial condition at time 0, $\sum_{i}Y_{i}^{+}(t)\geq\sum_{i}Y_{i}^{0}(t)$
for all $t$; this is intuitive since every customer in $\mathcal{N}^{+}$ keeps
the same service time and routing (relative to $\mathcal{N}^{0}$), but the
departure times must occur later due to the vacation periods.

%\textbf{Remark:} Once again, we observe that $\mathcal{N}$, $\mathcal{N}^{0}$,
%and $\mathcal{N}^{+}$ share the same arrival process and the protocol of
%assigning the service requirements and the routing directions remains the
%same. Nevertheless, as noted at the end of the construction of $\mathcal{N}%
%^{0}$, the customers might be routed differently in all of these networks.

\subsection{The Autonomous System}

The final bounding system is a set of the so-called autonomous queues which we
shall denote by $\mathcal{N}^{\prime}$. In this subsection, we shall describe
the evolution of this system and provide an expression for its number of
customers in queue. In the next subsection, we shall explain how
$\mathcal{N}^{\prime}$ is coupled with $\mathcal{N}^{+}$.
%some of the elements (such as the
%external arrival sequence) will be clearly related to the counterparts in
%network $\mathcal{N}^{+}$, others might not be as clear at first sight, but
%once we describe the evolution of $\mathcal{N}^{\prime}$ we will discuss the
%coupling with the elements in $\mathcal{N}^{+}$.

Define $\left(  N_{i}\left(  t\right)  :t\geq0\right)  $ to be the non-delayed
renewal process corresponding to the sequence $\{A_{i}\left(  n\right)
:n\geq1\}$; that is, defining $A_{i}\left(  0\right)  =0$, by convention we
have
\[
N_{i}\left(  t\right)  =\max\{n\geq0:A_{i}\left(  n\right)  \leq t\}.
\]
Of course, $N_{i}\left(  t\right)  \equiv0$ if $\lambda_{i}=0$.

We let $\left(  V_{i}^{0}\left(  k\right)  :k\geq1\right)  $ be a sequence of
i.i.d. random variables with the same distribution $\sigma_{i}^{0}\left(
k\right)  $ (and therefore as $\upsilon_{i}^{0}\left(  k\right)  $). We write
$B_{i}\left(  0\right)  =0$ and set $B_{i}\left(  n\right)  =V_{i}^{0}\left(
1\right)  +...+V_{i}^{0}\left(  n\right)  $. Then, define a renewal process%
\[
D_{i}\left(  t\right)  =\max\{n\geq0:B_{i}\left(  n\right)  \leq t\}.
\]
Moreover, for each $i\in\{1,...,d\}$ we define a sequence of i.i.d. random
variables $(r_{i}^{\prime}\left(  k\right)  :k\geq1)$ such that
\[
P\left(  r_{i}^{\prime}\left(  k\right)  =j\right)  =Q_{i,j},
\]
for all $j\in\{0,1,...,d\}$. We then define
\[
D_{i,j}\left(  t\right)  =\sum_{k=1}^{D_{i}\left(  t\right)  }I\left(
r_{i}^{\prime}\left(  k\right)  =j\right)  (\text{so that } D_{i}=\sum
_{j=0}^{d}D_{i,j}) .
\]

The random variables $V_{i}^{0}\left(  k\right)  $'s and $r_{i}^{\prime
}\left(  k\right)  $'s are all mutually independent and independent of the
$A_{i}\left(  k\right)  $'s for all $i\in\{1,...,d\}$ and $k\geq1$.

Let $Y_{i}^{\prime}\left(  t\right)  $ be the number of customers in the queue
at the $i$-th station of $\mathcal{N}^{\prime}$.
%It is important to note that $Y_{i}^{\prime}$ does \textit{not}
%include the potential customer in service, whereas the corresponding
%quantities $Y_{i},$ $Y_{i}^{0}$, and $Y_{i}^{+}$ are all numbers in system
%(i.e. queue and server).
By the definition of autonomous queues, $Y_{i}^{\prime}\left(  \cdot\right)  $
evolves according to the following Stochastic Differential Equation
\begin{align}
\text{d}Y_{i}^{\prime}\left(  t\right)   &  =\text{d}N_{i}\left(  t\right)
+\sum_{j:j\neq i,1\leq j\leq d}\text{d}D_{j,i}\left(  t\right)  -I\left(
Y_{i}^{\prime}\left(  t_{-}\right)  >0\right)  \text{d}D_{i}\left(  t\right)
,\label{SDE}\\
Y_{i}^{\prime}\left(  0\right)   &  =0.\nonumber
\end{align}
In simple words, the number of customers in queue at the $i$-th station
increases when there is an external arrival (d$N_{i}\left(  t\right)  =1$) or
an arrival (either virtual or true, see the explanation in Section
\ref{Section_Coupling}) from any other station ($\sum_{j=0}^{d}$%
d$D_{j,i}\left(  t\right)  =1$), and it decreases at time $t$ after the
completion of an activity (service or vacation, see the explanation is Section
\ref{Section_Coupling}) only if the queue is not empty (i.e. $I\left(
Y_{i}^{\prime}\left(  t_{-}\right)  >0\right)  $ and d$D_{i}\left(  t\right)
=1$).

One nice property of $\mathcal{N}^{\prime}$ is that we have a convenient
expression for $Y^{\prime}_{i}(t)$, which is essential for our CFTP algorithm
to work. Let's define
\[
X_{i}\left(  t\right)  =N_{i}\left(  t\right)  +\sum_{j:j\neq i,1\leq j\leq
d}D_{j,i}\left(  t\right)  -D_{i}\left(  t\right)  ,
\]
recall that $Q_{i,i}=0$ so we have that $D_{i,i}\left(  t\right)  =0$, and
thus we also can write $\sum_{j=1}^{d}D_{j,i}\left(  t\right)  $ in the
previous display. Then, one can verify that the (unique)\ solution to equation
(\ref{SDE}) is given by (see for instance, \cite{Harrison&Reiman})
\[
Y_{i}^{\prime}\left(  t\right)  =X_{i}\left(  t\right)  -\min_{0\leq s\leq
t}X_{i}\left(  s\right)  =\max_{0\leq s\leq t}\left(  X_{i}\left(  t\right)
-X_{i}\left(  s\right)  \right)  .
\]

\subsection{Coupling between $\mathcal{N}^{\prime}$ and $\mathcal{N}^{+}%
$\label{Section_Coupling}}

In order to describe the coupling between $\mathcal{N}^{\prime}$ and
$\mathcal{N}^{+}$, let us provide an interpretation of the SDE \eqref{SDE}
describing $\mathcal{N}^{\prime}$. The evolution of the $i$-th queue in
$\mathcal{N}^{\prime}$ can be seen as a single server queue with vacation
periods. Customers arrive according to the superposition of the processes
$N_{i}$ and $(D_{j,i}:1\leq j\leq d)$, the server takes a vacation whenever
the queue is empty with a distribution which is identical to that of a generic
service time. Arriving customers who find the queue empty must wait to be
served only until the current vacation epoch finishes.

The difference between $\mathcal{N}^{\prime}$ and $\mathcal{N}^{+}$ is that in
$\mathcal{N}^{+}$ no customers are \textquotedblleft
transferred\textquotedblright\ from station $i$ to $j$ at the end of a vacation epoch of server $i$.
Note that these types of transfers actually might occur in $\mathcal{N}%
^{\prime}$ because it could be the case, for instance, that $Y_{i}^{\prime
}\left(  t-\right)  =0$, $dD_{i}(t)=1$ and the corresponding $r_{i}^{\prime
}(k)=j$ so that $dD_{i,j}(t)=1$ and a new customer joins the queue at station
$j$. Consequently, in $\mathcal{N}^{\prime}$ there are two types of customers:
a) true customers, as those in $\mathcal{N}^{+}$, which are the ones that
correspond to external arrivals (i.e. arrivals from the processes $N_{i}$ for
$i\in\{1,...,d\}$), and their corresponding routes through the network, and b)
virtual customers, which does not exist in $\mathcal{N}^{+}$, are the ones
generated by empty stations that transfer customers to other stations by the
mechanism just described above. Therefore, to couple $\mathcal{N}^{\prime}$
and $\mathcal{N}^{+}$, we essentially need to distinguish between the true and
virtual customers in $\mathcal{N}^{\prime}$.

Recall that the evolution of $\mathcal{N}^{\prime}$ is fully described by the
process $N_{i}(\cdot)$, $D_{i}(\cdot)$ and $D_{ij}(\cdot)$, and $\mathcal{N}%
^{+}$ by the sequences $\{A_{i}(n)\}$, $\{r_{i}(k),\sigma_{i}^{0}(k)\}$ and
$\{v_{i}(k)\}$. To describe the coupling of $\mathcal{N}^{\prime}$ and
$\mathcal{N}^{+}$, we shall explain how to couple the pair of sequences.
Roughly speaking, the two systems will share the same external arrivals, and
each $V_{i}^{0}(k)$ (recall that $\{V_{i}^{0}(k)\}$ are the inter-renewal
times of $D_{i}(\cdot)$) corresponds to a service time $\sigma_{i}(k^{\prime
})$ when a customer is in service and to a vacation period $v_{i}(k^{\prime})$
otherwise. We provide the details next.

In our algorithm, we shall first simulate $\mathcal{N}^{\prime}$ on some
finite time interval $[T_{1},T_{2}]$, the corresponding processes $N_{i}(t)$,
$D_{i}(t)$ and $D_{ij}(t)$ on it, and sequences $\{A_{i}(n)\}$ and
$\{(V_{i}(k),r_{i}^{\prime}(k)\}$. Then, the number of customers $Y_{i}%
^{+}(t)$ of the coupled vacation system $\mathcal{N}^{+}$ evolves according to
the following SDE:
\begin{align}
\text{d}\widehat{Y}_{i}^{+}\left(  t\right)   &  =\text{d}N_{0,i}\left(
t\right)  +\sum_{j:j\neq i,1\leq j\leq d}I(S_{j}^{+}\left(  t_{-}\right)
>0)\text{d}D_{j,i}\left(  t\right)  -I(\widehat{Y}_{i}^{+}\left(
t_{-}\right)  >0)\text{d}D_{i}\left(  t\right)
,,\nonumber\label{Eq:Evolution_Vacation}\\
\text{d}\widehat{S}_{i}^{+}\left(  t\right)   &  =(I(\widehat{Y}_{i}%
^{+}\left(  t_{-}\right)  >0)-I(S_{i}^{+}\left(  t_{-}\right)  >0))\text{d}%
D_{i}\left(  t\right)  ,\\
Y_{i}^{+}(t)  &  =\widehat{Y}_{i}^{+}(t)+S_{i}^{+}(t).\nonumber
\end{align}
Here $S_{i}^{+}(t)\in\{0,1\}$ is the number of customer in service at station
$i$ at time $t$. In particular, we shall choose a special initial condition
for $\mathcal{N^{+}}$ according to the comparison results that we shall explain in
Section \ref{Sec:comparison}:
\begin{equation}
\widehat{Y}_{i}^{+}\left(  T_{1}\right)  =Y_{i}^{\prime}(T_{1}),\text{ }%
S_{i}^{+}\left(  T_{1}\right)  =1. \label{Eq:InitialCondition}%
\end{equation}
The remaining service time of the customer at station $i$ is the residual jump
time of $D_{i}(\cdot)$, i.e., $=B_{i}(D_{i}(T_{1})+1)-T_{1}$. Then, the
sequences of $(\sigma_{i}^{0}\left(  n\right)  $, $r_{i}\left(  n\right)
$)$_{n\geq1}$ can be extracted as follows.

\textbf{Procedure 0: Coupling of $\mathcal{N}^{\prime}$ and $\mathcal{N}^{+}%
$:}

\begin{enumerate}
[1)]

\item Input $N_{i}(t)$, $D_{i}(t)$ and $D_{i,j}(t)$ for $1\leq i,j\leq d$ and
$t\in[T_{1},T_{2}]$. Set $t_{i}=T_{1}$, $k^{s}_{i}=0$, $k^{v}_{i}=0$, and
$n_{i}=1$.

\item Compute $Y^{+}(t)$ and $S^{+}(t)$ according to
\eqref{Eq:Evolution_Vacation} and the initial condition \eqref{Eq:InitialCondition}.

\item For each $i$, while $t_{i}<T_{2}$, repeat the following:

\begin{itemize}
\item $t_{i}\leftarrow t_{i}+V^{0}_{i}(n)$;

\item If $S_{i}^{+}(t_{i}-)=1$, update $k^{s}_{i}\leftarrow k^{s}_{i}+1$ and
set $\sigma_{i}^{0}(k^{s}_{i})=V^{0}_{i}(n_{i})$ and $r_{i}(k^{s}_{i}%
)=r_{i}^{\prime}(n_{i})$. Otherwise, update $k^{v}_{i}\leftarrow k^{v}_{i}+1$
and set $v_{i}(k^{v}_{i})=V^{0}_{i}(n_{i})$.

\item $n_{i}\leftarrow n_{i}+1$.
\end{itemize}
\end{enumerate}

%\textcolor{blue}{

\begin{lem}
The extracted $(\sigma_{i}^{0}(k), r_{i}(k))$ form an i.i.d. sequence and
independent of the sequence $\{A_{i}(n)\}$.
\end{lem}

\begin{proof}
This follows from the strong Markov property of the forward recurrence time
processes of the renewal processes $N_{i}(\cdot)$, $D_{i}(\cdot)$ and
$D_{ij}(\cdot)$.
\end{proof}

%}

\subsection{Comparison Results and Domination}

\label{Sec:comparison}

Now we have a full description of the three systems $\mathcal{N}^{0}$,
$\mathcal{N}^{+}$ and $\mathcal{N}^{\prime}$ that are coupled with the
original GJN $\mathcal{N}$, and their corresponding queue length processes.
The following theorem gives the comparison results among the four systems,
which are essential in our DCFTP algorithm. Its proof is given in the Appendix.

\begin{thm}
\label{Thm_Domination}Suppose that the networks $\mathcal{N}$, $\mathcal{N}%
^{0}$, $\mathcal{N}^{+}$, and $\mathcal{N}^{\prime}$ are all initially empty
and are coupled as described through Section 4.1 to 4.4, then the following holds:

i) For any $t>0$,%
\[
\sum_{i=1}^{d}Y_{i}(t)\leq\sum_{i=1}^{d}Y_{i}^{0}(t)\leq\sum_{i=1}^{d}%
Y_{i}^{+}(t).
\]

ii) Moreover, for any $t>0$, when $Y_{i}^{\prime}(t)=y_{i}$, then the service
station $i$ in system $\mathcal{N}^{+}$ must satisfy $Y_{i}^{+}(t)\leq
y_{i}+1$ and $S_{i}^{+}(t)\in\{0,1\}$.

iii)\ The network $\mathcal{N}^{+}$, driven by the SDE
(\ref{Eq:Evolution_Vacation}), is monotone in the initial condition. In other
words, $y_{i}\in\{0,1,...\}$ and if $Y^{++},$ $Y^{+},$ $Y^{+-}$ satisfy the
SDEs (\ref{Eq:Evolution_Vacation}) with initial conditions $\widehat{Y}%
_{i}^{++}\left(  0\right)  =y_{i}+1$, $\widehat{S}_{i}^{++}\left(  0\right)
=1$; $\widehat{Y}_{i}^{+}\left(  0\right)  \leq y_{i}$, $\widehat{S}_{i}%
^{+}\left(  0\right)  \in\{0,1\}$, and $\widehat{Y}_{i}^{+}\left(  0\right)
=0=\widehat{S}_{i}^{+}\left(  0\right)  $, then $Y_{i}^{++}(t)\geq Y_{i}%
^{+}(t)\geq Y_{i}^{+-}(t)$ for all $t\geq0$.
\end{thm}

In the next section we explain how to use the previous result order to sample
from the stationary distribution of $\mathcal{N}$, i.e. the joint distribution of customer numbers at
each station, the remaining service requirement of the customers in service,
and the remaining times to the next external arrivals
to each station in steady state.

\section{Our Algorithm and Main Result}

Given the comparison results Theorem \ref{Thm_Domination}, we are now ready to
given the main procedure of our DCFTP algorithm. In the rest of the paper, for
any ergodic stochastic process $X(\cdot)$, we shall denote by $\bar{X}(\cdot)$
its two-sided stationary version.

\textbf{Main Procedure:}

\begin{enumerate}
\item Choose a constant $C_{T}>0$. Initialize $T\longleftarrow0$.

\item Simulate the system $\mathcal{N}^{\prime}$ in steady state and backwards in
time from $-T$ until $-T-C_{T}$. Obtain the corresponding processes $\bar
{N}_{i}(\cdot)$, $\bar{D}_{i}(\cdot)$, $\bar{D}_{ij}(\cdot)$ and $\bar
{Y}^{\prime}(\cdot)$ from $-T$ to $-T-C_{T}$. Update $T\longleftarrow T+C_{T}$.

\item Initialize a vacation system $\mathcal{N}^{++}$ at time $-T$ with
$Y_{i}^{++}(-T)=\bar{Y}_{i}^{\prime}(-T)+1$, all servers occupied ($S_{i}%
^{+}(-T)=1$) , and the corresponding remaining service time equals to the time
from $-T$ to the next jump time of process $\bar{D}_{i}(\cdot)$.

\item Compute $(Y_{i}^{++}(0):0\leq s\leq T)$, forward in time according to
\eqref{Eq:Evolution_Vacation} in Section \ref{Section_Coupling} and compute
the corresponding sequences $\{A_{i}(n)\}$, $\{r_{i}(k),\sigma^{0}_{i}(k)\}$
and $\{v^{0}_{i}(k)\}$ according to Procedure 0.

\item If there exists $\tau\in\lbrack0,T]$ such that $Y_{i}^{++}(\tau)=0$ for
all $i$, then we simulate a GJN $\mathcal{N}$ forward starting from $\tau<0$
to time $0$ with $Y_{i}(\tau)=0$ for all $i$ and driven by the sequence
$\{A_{i}(n)\}$ and $\{r_{i}(k),\sigma_{i}(k)\}$ where each $\sigma
_{i}(k)=\sigma^{0}_{i}(k)/a_{i}$. Output $Y_{i}(0)$ and terminate.
%then we must have that $\bar{Y}(-\left(
%T-T^{\prime\prime}\right)  )=0$, so we have all the information required to
%evolve $\bar{Y}$ forward in time until time zero and output $\bar{Y}\left(
%0\right)  $. (We must keep using the same sequence of arrivals, service times
%and routing directions as those used for $\mathcal{N}^{+}$ in Step 4.)

\item Otherwise, (if $Y^{+}(t)\neq0$\ for $t\in\lbrack0,T]$), go back to Step 2.
\end{enumerate}

The above procedure can be validated by the following heuristic. Suppose
$\bar{N}^{+}$ is the stationary vacation system coupled with $\bar{N}^{\prime
}$. Then, according to Part ii) and iii) of Theorem \ref{Thm_Domination}, its
queue length process $\bar{Y}_{i}^{+}(t)\leq Y_{i}^{++}(t)$ for all $i$ and
$t\in[-T,0]$. Therefore, we can conclude that $\bar{Y}_{i}(\tau)=0$ for all
$i$ and hence the coupled stationary GJN $\mathcal{N}$ must be empty at time
$\tau$ by Part i) of Theorem \ref{Thm_Domination}. Then, we can recover the
value of the stationary process $\bar{Y}_{i}(t)$ for $t\in[\tau, 0]$ and the
output $\bar{Y}_{i}(0)$ follows the steady-state distribution.

\begin{thm}
\label{Thm_Main}The state of the network given by the Main Procedure,
including $Y\left(  0\right)  $ and the remaining service times at each
station, follow the stationary distribution of the target GJN. Moreover, let
$N$ be the total number of random variables to terminate the Main Procedure,
then there is $\delta>0$ such that $E\exp\left(  \delta N\right)  <\infty$.
\end{thm}

Step 3 through Step 5 in the Main Procedure can be done according to the
coupling mechanism described in Section 3.1, 3.2, 3.4, and in particular,
Procedure 0. The most difficult part is the execution of Step 2 and we shall
explain this in Section \ref{Section_Step_2}. The proof, which is given at the
end of Algorithm \ref{algo last}, in Section \ref{Sec_Proof_Main}, mainly
constitutes a recapitulation of our development.

\section{Execution of Step 2 in Main Procedure: Stationary Construction and
Backward Simulation of $\mathcal{\bar{N}}^{\prime}$\label{Section_Step_2}}

This section is devoted to explain how to execute Step 2 in Main Procedure,
that is, to simulate a stationary version of $Y^{\prime}$ backwards in time.
We shall explain this simulation procedure in three steps. In Section
\ref{Section_Step_2}.1, we show a stationary version of $Y^{\prime}$ can be
expressed by a multi-dimensional point process and its maximum. Then, we show
the to simulate the point process and its maximum can be reduced to simulating
several random walks jointly with their maximum. In the end, in Section
\ref{Section_Step_2}.3, we explain how to simulate the random walks and their
maximum, following the ideas in \cite{Blanchet&Chen}.

%We first explain how to construct $\mathcal{\bar{N}}^{\prime}$ and then we
%indicate how to sample $\bar{Y}^{\prime}$ backwards in time. Once the
%construction of $\bar{Y}^{\prime}$ is performed, the fact that the algorithm
%terminates in finite time with probability, and that the GJN has a unique
%stationary version (see \cite{Sigman}), imply, owing to the validity of DCFTP
%(\cite{Kendall})\ that the outcome of Algorithm 1 has the correct stationary
%distribution of $\mathcal{N}$.

%\subsection{Notation and Basic Elements Behind $\mathcal{\bar{N}}^{\prime}$}

\subsection{Express $Y^{\prime}$ by Point Processes}

For each $i$, we define $\bar{N}_{0,i}\left(  \cdot\right)  $ as a two-sided,
time stationary, renewal point process with inter-arrival time distribution
being i.i.d. copies of $A_{i}\left(  n+1\right)  -A_{i}\left(  n\right)  $. We
write $\{\bar{A}_{i}\left(  n\right)  :n\in\mathbb{N}_{0}\mathbb{\cup(-N)\}}$
for the arrival times associated to $\bar{N}_{0,i}\left(  \cdot\right)  $, so
that $\bar{A}\left(  -1\right)  <0<\bar{A}\left(  0\right)  <\bar{A}\left(
1\right)  $ and define
\[
\bar{N}_{0,i}\left(  [a,b]\right)  =\sum_{n}I\left(  \bar{A}\left(  n\right)
\in\lbrack a,b]\right)  ,
\]
for any $a,b\in\left(  -\infty,\infty\right)  $.

Similarly, we let $\bar{D}_{i}\left(  \cdot\right)  $ to be a two-sided, time-stationary version of $D_{i}\left(  \cdot\right)  $ and write $\{\bar{B}%
_{i}\left(  n\right)  :n\in\mathbb{N}_{0}\mathbb{\cup(-N)\}}$ for the arrival
times associated to $\bar{D}_{i}\left(  \cdot\right)  $ also in increasing
order and so that $\bar{D}_{i}\left(  -1\right)  <0<\bar{D}_{i}\left(
0\right)  <\bar{D}_{i}\left(  1\right)  $. As before,
\[
\bar{D}_{i}\left(  [a,b]\right)  =\sum_{n}I\left(  \bar{B}_{i}\left(
n\right)  \in\lbrack a,b]\right)  .
\]
Each $\bar{B}_{i}\left(  n\right)  $ is attached to a mark $\bar{r}%
_{i}^{\prime}\left(  n\right)  $ which are i.i.d. copies of the $r_{i}%
^{\prime}\left(  n\right)  $'s. All the $\bar{A}_{i}\left(  n\right)  $'s, the
$\bar{D}_{i}\left(  n\right)  $'s, and the $r_{i}^{\prime}\left(  n\right)
$'s are mutually independent. Finally, for any $a, b \in (-\infty,\infty)$, define
\[
\bar{D}_{i,j}\left(  [a,b]\right)  =\sum_{n}I\left(  \bar{B}_{i}\left(
n\right)  \in\lbrack a,b],\bar{r}_{i}^{\prime}\left(  n\right)  =j\right)  .
\]

Intuitively, $N_{0,i}(\cdot)$ describes the external arrivals to station $i$,
$D_{i,0}(\cdot)$ describes the potential departures from station $i$, and
$D_{ij}(\cdot)$ describes the potential internal routings from station $i$ to
$j$. For all $t\geq0$, we define
\begin{align}
\bar{N}_{0,i}\left(  t\right)   &  =\bar{N}_{0,i}\left(  [0,t]\right)  ,\text{
}\bar{N}_{0,i}\left(  -t\right)  =-\bar{N}_{0,i}\left(  [-t,0)\right)
,\label{Eq:Two_Sided_Processes}\\
\bar{D}_{i}\left(  t\right)   &  =\bar{D}_{i}\left(  [0,t]\right)  ,\text{
}\bar{D}_{i}\left(  -t\right)  =-\bar{D}_{i}\left(  [-t,0)\right)
,\nonumber\\
\bar{D}_{i,j}\left(  t\right)   &  =\bar{D}_{i,j}\left(  [0,t]\right)  ,\text{
}\bar{D}_{i,j}\left(  -t\right)  =-\bar{D}_{i,j}\left(  [-t,0)\right)
.\nonumber
\end{align}
and
\[
\bar{X}_{i}\left(  t\right)  =\bar{N}_{0,i}\left(  t\right)  +\sum_{j:j\neq
i,1\leq j\leq d}\bar{D}_{j,i}\left(  t\right)  -\bar{D}_{i}\left(  t\right)
.
\]
Then, $\bar{X}_{i}\left(  t\right)  $ is a two-sided stationary process.
Finally put for $t\leq0$,
\begin{equation}
\bar{Y}^{\prime}(-t)=-\bar{X}(t)+\sup_{s\geq t}\bar{X}(s).
\label{Stat_version}%
\end{equation}
Observe that the for any deterministic time $T<0$, the process process
$\{\bar{Y}_{i}^{\prime}\left(  T+t\right)  -\bar{Y}_{i}^{\prime}\left(
T\right)  :0\leq t\leq\left\vert T\right\vert \}$ satisfies the SDE
(\ref{SDE}) only replacing the renewal processes with their respective
stationary versions. We just need to show that $Y^{\prime}$ has a unique
stationary distribution which is the same as the distribution of $\bar
{Y}^{\prime}\left(  0\right)  $ and thus we have that $\bar{Y}^{\prime}$ is
the time-reversed, stationary version of $Y^{\prime}$.

\begin{lem}
\label{Lemma_Lyones_Auto}The autonomous queue $Y^{\prime}(\cdot)$ has a unique
stationary distribution and therefore $\{\bar{Y}^{\prime}(-t):t\geq0\}$ given
by (\ref{Stat_version}) is the time-reversed, stationary version of
$Y^{\prime}$.
\end{lem}

\textit{Proof of Lemma \ref{Lemma_Lyones_Auto}}: We proceed with a
construction procedure similar to the Loynes method. For $t\in\lbrack0,T]$ and any
$y\in\mathbb{R}^{d}$ define
\begin{align*}
\text{d}Y_{i}^{\prime}\left(  t\right)   &  =\text{d}N_{0,i}\left(  t\right)
+\sum_{j:j\neq i,1\leq j\leq d}\text{d}D_{j,i}\left(  t\right)  -I\left(
Y_{i}^{\prime}\left(  t_{-}\right)  >0\right)  \text{d}D_{i}\left(  t\right)
,\\
Y_{i}^{\prime}\left(  0\right)   &  =y.
\end{align*}
We then have that%
\[
Y_{i}^{\prime}(t)=(y_{i}+X_{i}(t))-\inf_{0\leq s\leq t}\min(y_{i}+X_{i}(s),0)
\]
and therefore%
\begin{align*}
Y_{i}^{\prime}(T)  &  =X_{i}(T)-\inf_{0\leq s\leq T}\min(X_{i}(s),-y_{i})\\
&  =-\inf_{0\leq s\leq T}\{\min(X_{i}(s),-y_{i})-X_{i}\left(  T\right)  \}\\
&  =\sup_{0\leq s\leq T}\{\max(X_{i}\left(  T\right)  -X_{i}\left(  s\right)
,y_{i}+X_{i}\left(  T\right)  )\}\\
&  =\sup_{0\leq u\leq T}\{\max(X_{i}\left(  T\right)  -X_{i}\left(
T-u\right)  ,y_{i}+X_{i}\left(  T\right)  )\}.
\end{align*}
As $T\rightarrow\infty$ we have that $X_{i}\left(  T\right)  \rightarrow
-\infty$ and $X_{i}\left(  T\right)  -X_{i}\left(  T-u\right)  \Rightarrow$
$\bar{X}_{i}\left(  u\right)  $ as $T\rightarrow\infty$ (weakly)\ and
therefore $Y_{i}^{\prime}(T)\Rightarrow\bar{Y}_{i}^{\prime}\left(  0\right)  $
regardless of the initial condition.\hfill$\Box$

Given the time-reversed, stationary version of $Y^{\prime}$, it suffices to
simulate
\[
\bar{X}_{i}^{\ast}\left(  t\right)  =\sup_{r\geq t}\bar{X}_{i}\left(
r\right)  ,
\]
jointly with $\bar{X}_{i}\left(  t\right)  $ for all $i\in\{1,...,d\}$.

\subsection{Connection between $\bar{X}_{i}^{\ast}\left(  t\right)  $ and
Associated Random Walks}

We note that $E[\bar{X}_{i}(1)]<1$
%\begin{align*}
%E\bar{X}_{i}\left(  t\right)   &  =E\bar{N}_{0,i}\left(  t\right)
%+\sum_{j:j\neq i,1\leq j\leq d}E\bar{D}_{j,i}\left(  t\right)  -E\bar{D}
%_{i}\left(  t\right) \\
%&  =\lambda_{i}t+\sum_{j=1}^{d}Q_{j,i}\mu_{j}^{0}t-\mu_{i}^{0}t<0,
%\end{align*}
due to (\ref{Eq:Stability_Bound}), therefore, $\bar{X}_{i}\left(  t\right)
\rightarrow-\infty$ as $t\rightarrow\infty$. Note that
\[
X_{i}^{\ast}\left(  t\right)  =\max(\sup\{X_{i}\left(  r\right)  :t\leq r\leq
u\},X_{i}^{\ast}\left(  u\right)  ).
\]
To construct a bound for $\bar{X}^{*}_{i}(\cdot)$, we will construct a
non-increasing process $Z_{i}\left(  \cdot\right)  $, such that $Z_{i}\left(
u\right)  \geq$ $X_{i}^{\ast}\left(  u\right)  $ and $Z_{i}\left(  u\right)
\rightarrow-\infty$ with probability one. Since $\sup\{X_{i}\left(  r\right)
:t\leq r\leq u\}$ is clearly non-decreasing in $u$, our ability to simulate
$Z_{i}\left(  u\right)  $ will allow us to sample $X^{\ast}\left(  t\right)  $
in finite time.

\subsubsection{Construction of the Upper Bound $Z_{i}\left(  \cdot\right)  $}

We now give the definition of $Z_{i}\left(  \cdot\right)  $. Following
(\ref{Eq:Stability_Bound}), we can pick $\bar{\delta}>0$ small enough so that
\[
\lambda_{i}+\sum_{j=1}^{d}Q_{j,i}\mu_{j}^{0}+\bar{\delta}\left(  1+\sum
_{j=1}^{d}Q_{j,i}\right)  <\mu_{i}^{0}.
\]
Next we define $\gamma_{i}=\lambda_{i}+\bar{\delta}$, $\varphi_{j,i}%
=Q_{j,i}(\mu_{j}+\bar{\delta})$ and $\beta_{i}=\gamma_{i}+\sum_{j=1}%
^{d}\varphi_{j,i}$, and split%

\[
\bar{X}_{i}\left(  t\right)  =\left(  \bar{N}_{0,i}\left(  t\right)
-\gamma_{i}t\right)  +\sum_{j=1}^{d}\left(  \bar{D}_{j,i}\left(  t\right)
-\varphi_{j,i}t\right)  +(\beta_{i}t-\bar{D}_{i}\left(  t\right)  )
\]
so that
\[
\bar{X}_{i}^{\ast}\left(  t\right)  \leq\sup_{r\geq t}\left(  \bar{N}%
_{0,i}\left(  t\right)  -\gamma_{i}t\right)  +\sum_{j=1}^{d}\sup_{r\geq
t}\left(  \bar{D}_{j,i}\left(  t\right)  -\varphi_{j,i}t\right)  +\sup_{r\geq
t}(\beta_{i}t-\bar{D}_{i}\left(  t\right)  ).
\]
Finally, we define three non-increasing processes as
\begin{align*}
\bar{N}_{0,i}^{\ast}\left(  t\right)   &  =\sup_{r\geq t}\left(  \bar{N}%
_{0,i}\left(  r\right)  -\gamma_{i}r\right)  ,\\
\bar{D}_{j,i}^{\ast}\left(  t\right)   &  =\sup_{r\geq t}\left(  \bar{D}%
_{j,i}\left(  r\right)  -\varphi_{j,i}r\right)  ,\\
\bar{D}_{i}^{\ast}\left(  t\right)   &  =\sup_{r\geq t}(\beta_{i}r-\bar{D}%
_{i}\left(  r\right)  ),
\end{align*}
for all $t\geq0$. Observe that by the selection of $\beta_{i}$, $\varphi
_{j,i}$, and $\gamma_{i}$, all the three processes just defined are
non-increasing and go to minus infinity with probability 1. As a result,
\[
Z_{i}\left(  t\right)  :=\bar{N}_{0,i}^{\ast}\left(  t\right)  +\sum_{j=1}%
^{d}\bar{D}_{j,i}^{\ast}\left(  t\right)  +\bar{D}_{i}^{\ast}\left(  t\right)
\text{ is non-increasing and goes to }-\infty\text{ as }t\to\infty.
\]

Now we explain how to simulate jointly
\[
(\bar{N}_{0,i}^{\ast}\left(  t\right)  ,\bar{N}_{0,i}\left(  t\right)
,\bar{D}_{j,i}^{\ast}\left(  t\right)  ,\bar{D}_{j,i}\left(  t\right)
,\bar{D}_{i}^{\ast}\left(  t\right)  ,\bar{D}_{i}\left(  t\right)
:i,j\in\{1,...,d\}).
\]

\subsubsection{Transforming the Simulation of $(Z\left(  t\right)  :t\geq0)$
into that of the Maximum of a Multidimensional Random Walk}

Note that $\bar{N}_{0i}(\cdot)$ is piecewise linear with jumps, therefore it
reaches its maximum only at (or right before) the times $\{\bar{A}_{i}(n)\}$
when it jumps. So are $D_{i}(\cdot)$ and $D_{ij}(\cdot)$. These results are
formalized by the following lemma:

\begin{lem}
\label{Lem_Aux_Eval}For $t\geq0$ and assuming that $Q_{i,j}>0$ in the case of
$\bar{D}_{j,i}^{\ast}\left(  t\right)  $, we have that
\begin{align*}
\bar{N}_{0,i}^{\ast}\left(  t\right)   &  =\max\left(  \bar{N}_{0,i}%
(t)-\gamma_{i}t, \sup_{n>\bar{N}_{0,i}\left(  t\right)  }\left(  n-\gamma
_{i}\bar{A}_{i}\left(  n\right)  \right)  +1 \right)  ,\\
\bar{D}_{i}^{\ast}\left(  t\right)   &  =\sup_{n>\bar{D}_{i}( t) }%
\Big(\beta_{i}\bar{B}_{i}\left(  n\right)  -n\Big),\\
\bar{D}_{j,i}^{\ast}\left(  t\right)   &  =\max\left(  \bar{D} _{j,i}\left(
t\right)  -\varphi_{j,i}t,\sup_{n>\bar{D}_{j}( t) }\left(  \sum_{k=1}^{n}I(
r_{j}^{\prime}\left(  k\right)  =i) -\varphi_{j,i}\bar{B}_{j}\left(  n\right)
\right)  +1 \right)  .
\end{align*}

\end{lem}

\textit{Proof of Lemma }\ref{Lem_Aux_Eval}: By definition, for any $r\geq0$
such that $\bar{A}_{i}(k)\leq r<\bar{A}_{i}(k+1)$, $\bar{N}_{0,i}(r)=k+1$. As
a result,
\[
\max_{\bar{A}_{i}(k)\leq r<\bar{A}_{i}(k+1)}\left(  \bar{N}_{0,i}%
(r)-\gamma_{i}r\right)  =k+1-\bar{A}_{i}(k),
\]
and the maximum is reached at $r=\bar{A}_{i}(k)$. As $\bar{A}_{i}(\bar
{N}_{0,i}(t))\leq t<\bar{A}_{i}(\bar{N}_{0,i}(t)+1)$,
\[
\bar{N}_{0,i}^{\ast}(t)=\sup_{r\geq t}(\bar{N}_{0,i}(r)-\gamma_{i}%
r)=\sup_{t\leq r<\bar{A}_{i}(\bar{N}_{0,i}(t)+1)}(\bar{N}_{0,i}(r)-\gamma
_{i}r)\vee~\sup_{n>\bar{N}_{0,i}(t)}(n+1-\gamma_{i}\bar{A}_{i}(n)).
\]
As
\[
\sup_{t\leq r<\bar{A}_{i}(\bar{N}_{0,i}(t)+1)}(\bar{N}_{0,i}(r)-\gamma
_{i}r)=\bar{N}_{0,i}(t)-\gamma_{i}t,
\]
and
\[
\sup_{n>\bar{N}_{0,i}(t)}(n+1-\gamma_{i}\bar{A}_{i}(n))=\sup_{n>\bar{N}%
_{0,i}\left(  t\right)  }\left(  n-\gamma_{i}\bar{A}_{i}\left(  n\right)
\right)  +1,
\]
we have reach the expression for $\bar{N}_{0,i}^{\ast}(t)$. The same argument
applies to $\bar{D}_{j,i}^{\ast}(\cdot)$. As to $\bar{D}_{i}^{\ast}(\cdot)$,
note that
\[
\sup_{\bar{B}_{i}(k)\leq r<\bar{B}_{i}(k+1)}(\beta_{i}r-\bar{D}_{i}%
(r))=\beta_{i}B_{i}(k+1)-(k+1)=\lim_{r\rightarrow\bar{B}_{i}(k+1)-}\beta
_{i}r-\bar{D}_{i}(r),
\]
and $\bar{B}_{i}(\bar{D}_{i}(t))\leq t<\bar{B}_{i}(\bar{D}_{i}(t)+1)$,
therefore $\bar{D}_{i}^{\ast}(t)=\sup_{n>\bar{D}_{i}(t)}(\beta_{i}\bar{B}%
_{i}(n)-n)$. \hfill$\Box$

Therefore, to simulate the processes $\bar{N}^{*}_{0,i}(\cdot), \bar{D}%
^{*}_{i}(\cdot)$ and $\bar{D}^{*}_{ij}(\cdot)$, we only need to observe the
processes $\bar{N}_{0,i}(\cdot), \bar{D}_{i}(\cdot)$ and $\bar{D}_{ij}(\cdot)$
at the discrete times when they jump, which can be expressed as random walks.
The random walks have increments $(\bar{U}_{0,i}(n), \bar{V}_{i}^{0}, \bar
{V}^{0}_{ji}(0))$ defined as
\begin{align*}
\bar{U}_{0,i}\left(  n\right)   &  =1-\gamma_{i}(\bar{A}_{i}\left(  n\right)
-\bar{A}_{i}\left(  n-1\right)  ),\text{ \ }\bar{V}_{i}^{0}\left(  n\right)
=\beta_{i}(\bar{B}_{i}\left(  n\right)  -\bar{B}_{i}\left(  n-1\right)  )-1,\\
\bar{V}_{j,i}^{0}\left(  n\right)   &  =I(r_{j}^{\prime}\left(  n\right)
=i)-\varphi_{j,i}(\bar{B}_{j}\left(  n\right)  -\bar{B}_{j}\left(  n-1\right)
), \text{ for }n\geq1
\end{align*}
and for $n=0$,
\[
\bar{U}_{i}\left(  0\right)  =-\gamma_{i}\bar{A}_{i}\left(  0\right)  ,\text{
\ \ }\bar{V}_{i}^{0}\left(  0\right)  =\beta_{i}\bar{B}_{i}\left(  0\right)
,\text{ \ \ }\bar{V}_{j,i}^{0}\left(  0\right)  =-\varphi_{j,i}\bar{B}%
_{i}\left(  0\right)  .
\]
For the pair of $(i,j)$ with $\varphi_{j,i}=0$, we have that $\bar{V}%
_{j,i}^{0}\left(  n\right)  \equiv0$ and we can ignore these coordinates. But
in order to keep the notation succinct, let us denote by $$\bar{W}%
_{i}\left(  n\right)  =(\bar{U}_{i}\left(  n\right)  ,\bar{V}_{i}^{0}\left(
n\right)  ,\bar{V}_{1,i}^{0}\left(  n\right)  ,...,\bar{V}_{d,i}^{0}\left(
n\right)  )^{T}$$ for $n\geq0$, and let
\[
W\left(  n\right)  =(\bar{W}_{1}\left(  n\right)  ,\bar{W}_{2}\left(
n\right)  ,...,\bar{W}_{d}\left(  n\right)  )^{T}.
\]
Observe that $W\left(  n\right)  $ is a vector of dimension $d\times\left(
d+2\right)  $. To make the notation homogeneous we write $W_{j}\left(
n\right)  $ for the $j$-th coordinate of $W\left(  n\right)  $ where $1\leq
j\leq d\times(d+2)$. Now we can define a $d\times\left(  d+2\right)
$-dimensional random walk $S\left(  k\right)  =S\left(  k-1\right)  +W\left(
k\right)  $, for $k\geq1$, with $S\left(  0\right)  =W\left(  0\right)  $.
Define its maximum process as
\[
M_{j}\left(  n\right)  =\sup_{k\geq n}S_{j}\left(  k\right)  \text{ for }1\leq
j\leq d\times(d+2).
\]
Following Lemma\textit{ }\ref{Lem_Aux_Eval}, to simulate
\[
(\bar{N}_{0,i}^{\ast}\left(  t\right)  ,\bar{N}_{0,i}\left(  t\right)
,\bar{D}_{j,i}^{\ast}\left(  t\right)  ,\bar{D}_{j,i}\left(  t\right)
,\bar{D}_{i}^{\ast}\left(  t\right)  ,\bar{D}_{i}\left(  t\right)
:i,j\in\{1,...,d\})
\]
is equivalent to simulate $\left(  M\left(  n\right)  ,S\left(  n\right)
:n\geq0\right)  $ jointly. Fortunately, there is an algorithm that allows us
to carry out this simulation problem for $(M\left(  n\right)  ,S\left(
n\right)  )$, adapted from work of \cite{Blanchet&Chen} and
\cite{Blanchet&Wallwater}, we provide details here for completeness.

\textbf{Remark:} In the following sections we shall simulate $(M\left(
n\right)  -W\left(  0\right)  ,S\left(  n\right)  -W\left(  0\right)  )$,
which is equivalent to simulating the sequence $\left(  M\left(  n\right)
,S\left(  n\right)  :n\geq0\right)  $ assuming that $S\left(  0\right)  =0$.
In the end, the random variable $W\left(  0\right)  $ can be simulated
independently from everything else.

\subsection{Sampling the Infinite Horizon Maximum of a Multidimensional Random
Walk with Negative Drift}

Let $S_{l^{\prime}}\left(  n\right)  $ be the coordinate of the random walk
corresponding to $\bar{V}_{j,i}^{0}\left(  n\right)  $. We have that either
$S_{l^{\prime}}\left(  n\right)  \equiv 0$ when $Q_{i,j}=0$, or $E[S_{l^{\prime}%
}\left(  n\right) ] <0$. For those coordinates for which $S_{l^{\prime}}\left(
n\right)  \equiv 0$ we have that $M_{l^{\prime}}\left(  n\right)  =0$ and there is
nothing to do. So, let us assume for simplicity and without loss of generality
that $E[S_{i}\left(  n\right)]  <0$ for all $1\leq i\leq l=d\left(  d+2\right)
$.

Define for each $\theta\in\mathbb{R}$%
\[
\psi_{i}\left(  \theta\right)  =\log E[\exp\left(  \theta_{i}W_{i}\left(
k\right)  \right) ] ,
\]
and set
\begin{align}
&  P_{\theta_{i}}\left(  W_{1}\left(  k\right)  \in dy_{1},...,W_{l}\left(
k\right)  \in dy_{l}\right) \label{Eq:Exp_Tilt}\\
=~  &  \frac{\exp\left(  \theta_{i}y_{i}-\psi_{i}\left(  \theta\right)
\right)  }{E[\exp\left(  \theta_{i}W_{i}\left(  k\right)  \right)]  }P\left(
W_{1}\left(  k\right)  \in dy_{1},...,W_{d}\left(  k\right)  \in
dy_{d}\right)  ,\nonumber
\end{align}
where $\theta_{i}\in\mathbb{R}$ and $E\exp\left(  \theta_{i}W_{i}\left(
k\right)  \right)  <\infty$. Moreover, we impose the following assumption for simplicity.

\textbf{Assumption 2b): }For each $i$ there exists $\theta_{i}^{\ast}$ such
that
\[
\psi_{i}\left(  \theta_{i}^{\ast}\right)  =0.
\]

\textbf{Remark: }Assumption 2b) is a strengthening of Assumption 2. We can
carry out our ideas under Assumption 2 following \cite{Blanchet&Chen} as we
explain next. First, instead of ($M\left(  n\right)  :n\geq0)$, given a vector
$a^{\prime}=\left(  a_{1}^{\prime},a_{2}^{\prime},...,a_{d}^{\prime}\right)
^{T}$ with non-negative components that we will explain how to choose
momentarily, consider the process $S_{a^{\prime}}(\cdot)$ and $M_{a^{\prime}%
}\left(  \cdot\right)  $ defined by
\[
S_{a^{\prime}}\left(  n\right)  :=S\left(  n\right)  +a^{\prime}%
n,~M_{a^{\prime}}\left(  n\right)  =\max_{k\geq n}\left(  S_{a}\left(
k\right)  \right)  .
\]
Note that we can simulate $\left(  S\left(  n\right)  ,M\left(  n\right)
:n\geq0\right)  $ if we are able to simulate $\left(  S_{a^{\prime}}\left(
n\right)  ,M_{a^{\prime}}\left(  n\right)  :n\geq0\right)  $. Now, note that
$\psi_{i}\left(  \cdot\right)  $ is strictly convex and that $d\psi_{i}\left(
0\right)  /d\theta<0$ so there exists $a_{i}^{\prime}>0$ large enough to force
the existence of $\theta_{i}^{\ast}>0$ such that $E[\exp\left(  \theta
_{i}^{\ast}W_{i}\left(  1\right)  +a_{i}^{\prime}\theta_{i}^{\ast}\right)
]=1$, but at the same time small enough to keep $E[\left(  W_{i}\left(
1\right)  +a_{i}^{\prime}\right)]  <0$; again, this follows by strict convexity
of $\psi_{i}\left(  \cdot\right)  $ at the origin. So, if Assumption A3b) does
not hold, but Assumption A3) holds, one can then execute Algorithm
\ref{algo Sampling M_0} based on the process $S_{a^{\prime}}(\cdot)$.\textbf{
}

\subsubsection{Construction of $\left(  S\left(  n\right)  ,M\left(  n\right)
:\,n\geq0\right)  $ via \textquotedblleft milestone events\textquotedblright%
\label{sec:Construction S_n M_n}}

We will describe the construction of a pair of sequences of stopping times
(with respect to the filtration generated by $(S\left(  n\right)  :n\geq0)$),
denoted by $(\Lambda_{n}:n\geq0)$ and $(\Gamma_{n}:n\geq1)$, which track
certain downward and upward milestones in the evolution of $\left(  S\left(
n\right)  :\,n\geq0\right)  $. We follow similar steps as described in
\cite{Blanchet&Sigman} and \cite{Blanchet&Wallwater}. These \textquotedblleft
milestone events\textquotedblright\ will be used in the design of our proposed
algorithm. The elements of the two stopping times sequences interlace with
each other (when finite) and their precise description follows next.

%%\begin{SCfigure}[ptb]
%%\begin{minipage}[t]{0.5\columnwidth}%
%%\includegraphics[width=\textwidth]{path1.pdf}%
%%\end{minipage}
%%\begin{minipage}[t]{0.5\columnwidth}%
%%\caption{kskksksk}
%%\end{minipage}
%%\label{fig:Construction of stopping times}%
%%\end{SCfigure}

%%%%%%%%%%%%%%%%%%%%%%%%%%%%%%%%%%%%%%
%%\begin{SCfigure}
%%\label{fig:Construction of stopping times}%
%%\centering
%%\includegraphics[width=0.5\textwidth]%
%%{path2.pdf}% picture filename
%%\vspace{-27pt}
%%\caption{ \textbf{Figure \ref{fig:Construction of stopping times}}
%%Figure
%%\ref{fig:Construction of stopping times} illustrates a sample path $\left\lbrace S_n\left(\mu\right): \, 0\leq n\leq 12\right\rbrace$.
%%If we set $m=1$ and $L=2$  then the corresponding stopping times  are $D_{1}=4$, $U_{1}=6$, $D_{2}=9$.
%%If in addition $U_{2}=\infty$, then $S_{n}\left(\mu\right)$ stays below the doted bold line for all $n\geq D_2$.
%%Therefore, at time $t=D_2$ the values of $\left\lbrace M_n:0\leq n\leq 7\right\rbrace$ can be calculated and
%%we can update $C_{UB}\leftarrow S_{D_{2}%
%%}\left(  \mu\right)  +1$.
%%}
%%\end{SCfigure}
%%%%%%%%%%%%%%%%%%%%%%%%%%%%%%%%%%%FIGURE

We start by fixing any $m>0$. Eventually, we shall choose $m$ suitably large
as we shall discuss in in equation (\ref{Eq:Assumption_m}), but our conceptual
discussion here is applicable to any $m>0$. Now set $\Lambda_{0}=0$. We
observe the evolution of the process $(S\left(  n\right)  :\,n\geq0)$ and
detect the time $\Lambda_{1}$ (the first downward milestone),
\[
\Lambda_{1}=\inf\left\{  n\geq\Lambda_{0}:\,S(n)<-m\mathbf{e}\right\}  ,
\]
where the inequality is componentwise. That is, $S_{i}(n)<-m$ for all $1\leq
i\leq l$.

Once $\Lambda_{1}$ is detected we check whether or not $\left\{  S\left(
n\right)  :\,n\geq\Lambda_{1}\right\}  $ ever goes above the height $S\left(
\Lambda_{1}\right)  +m$ (the first upward milestone); namely we define
\[
\Gamma_{1}=\inf\left\{  n\geq\Lambda_{1}:\,S_{i}(n)>m+S_{i}\left(  \Lambda
_{1}\right)  \text{ for some }1\leq i\leq l\right\}  .
\]

For now let us assume that we can check if $\Gamma_{1}=\infty$ or $\Gamma
_{1}<\infty$ (how exactly to do so will be explained in Section
\ref{Sec_Sampling_M0}). To continue simulating the rest of the path, namely
$\left\{  S\left(  n\right)  :\,n>\Lambda_{1}\right\}  $, we potentially need
to keep track of the conditional upper bound implied by the fact that
$\Gamma_{1}=\infty$. To this end, we introduce the conditional upper bound
variable $C_{UB}$ (initially $C_{UB}=\infty$). If at time $\Lambda_{1}$ we
detect that $\Gamma_{1}=\infty$, then we set $C_{UB}=S\left(  \Lambda
_{1}\right)  +m$ and continue sampling the path of the random walk conditional
on never crossing the upper bound $S\left(  \Lambda_{1}\right)  +m$ in any of
the coordinates. That is, conditional on $\left\{  S\left(  n\right)
<C_{UB}:\,n>\Lambda_{1}\right\}  $. Otherwise, if $\Gamma_{1}<\infty$, we
simulate the path conditional on $\Gamma_{1}<\infty$, until we detect the time
$\Gamma_{1}$. We continue on, sequentially checking whenever a downward or an
upward milestone is crossed as follows: for $j\geq2$, define%

\begin{equation}%
\begin{array}
[c]{l}%
\Lambda_{j}=\inf\left\{  n\geq\Gamma_{j-1}I\left(  \Gamma_{j-1}<\infty\right)
\vee\Lambda_{j-1}:\,S\left(  n\right)  <S\left(  \Lambda_{j-1}\right)
-m\mathbf{e}\right\} \\
\Gamma_{j}=\inf\left\{  n\geq\Lambda_{j}:\,S_{i}\left(  n\right)
-S_{i}\left(  \Lambda_{j}\right)  >m\text{ for some }1\leq i\leq l\right\}  ,
\end{array}
\label{eq:construction of D_j and U_j}%
\end{equation}
with the convention that if $\Gamma_{j-1}=\infty$, then $\Gamma_{j-1}I\left(
\Gamma_{j-1}<\infty\right)  =0$. Therefore, we have that $\Gamma_{j-1}I\left(
\Gamma_{j-1}<\infty\right)  >\Lambda_{j-1}$ if and only if $\Gamma
_{j-1}<\infty$.

Let us define
\begin{equation}
\Delta=\inf\{\Lambda_{n}:\Gamma_{n}=\infty,n\geq1\}. \label{eq def Delta}%
\end{equation}
So, for example, if $\Gamma_{1}=\infty$ we have that $\Delta=\Lambda_{1}$ and
the drifted random walk will never reach level $S\left(  \Lambda_{1}\right)
+m<S(0)$ again. This allows us to evaluate $M\left(  0\right)  $ by computing
\begin{equation}
M\left(  0\right)  =\max\left\{  S\left(  n\right)  :\,0\leq n\leq
\Delta\right\}  , \label{Eq_EVAL_M0}%
\end{equation}
the maximum is taken over $n$ for each coordinate.

Similarly, the event $\Gamma_{j}=\infty$, for some $j\geq1$, implies that the
level $S_{i}\left(  \Lambda_{j}\right)  +m$ is never crossed for any $i$ (that
is $S_{i}\left(  n\right)  \leq S_{i}\left(  \Lambda_{j}\right)  +m$) for all
$n\geq\Lambda_{j}$, and we let $C_{UB}=S\left(  \Lambda_{j}\right)  +m$. The
value of the vector $C_{UB}$ keeps updating as the random walk evolves, at
times where $\Gamma_{j}=\infty$.

The advantage of considering these stopping times is the following: once we
observed that some $\Gamma_{j}=\infty$, the values of $\left\{  M_{i}\left(
n\right)  :\,n\leq\Gamma_{j-1}1(\Gamma_{j-1<\infty})\vee\Lambda_{j-1}\right\}
$ for each $1\leq i\leq l$ are known without a need of further simulation.
%
%%Figure <ref>fig:Construction of stopping times</ref>.
Proposition \ref{pro:behavior of Delta_n and Gamma_n} ensures that it suffices
to sequentially simulate $(\Lambda_{n}:n\geq0)$ and $(\Gamma_{n}:n\geq1)$
jointly with the underlying random walk in order to sample from the sequence
$\left(  S\left(  n\right)  ,\,M\left(  n\right)  :\,n\geq0\right)  $. The
proof of Proposition \ref{pro:behavior of Delta_n and Gamma_n} is easily
adapted from the one dimensional case discussed in \cite{Blanchet&Wallwater}
and thus it is omitted.

\begin{prop}
\label{pro:behavior of Delta_n and Gamma_n} Set $\Lambda_{0}=0$ and let
$(\Lambda_{n}:\,n\geq1)$ and $(\Gamma_{n}:\,n\geq1)$ be as
(\ref{eq:construction of D_j and U_j}). We have that
\begin{equation}%
\begin{array}
[c]{cccc}%
P_{0}\left(  \lim_{n\rightarrow\infty}\Lambda_{n}=\infty\right)  =1 & \text{
and } & P_{0}\left(  \Lambda_{n}<\infty\right)  =1,\, & \forall n\geq1.
\end{array}
\label{eq:behavior of D_n}%
\end{equation}
Furthermore,
\begin{equation}
P_{0}\left(  \Gamma_{n}=\infty,\,\text{i.o.}\,\right)  =1.
\label{eq: behavior U_n io as}%
\end{equation}

\end{prop}

\medskip{}

In the setting of Proposition \ref{pro:behavior of Delta_n and Gamma_n}, for
each $k\geq0$ we can define $N_{0}\left(  k\right)  =\inf\left\{
n\geq1:\,\Lambda_{n}\geq k\right\}  $ and $\mathcal{T}\left(  k\right)
=\inf\left\{  j\geq N_{0}\left(  k\right)  +1:\,\Gamma_{j}=\infty\right\}  $.
Both of them are finite random variables such that%
\begin{equation}
M\left(  k\right)  =\max\left\{  S\left(  n\right)  :\,k\leq n\leq
\Lambda_{\mathcal{T}\left(  k\right)  }\right\}
\label{eq:connection M_n and stopping times}%
\end{equation}

%\begin{equation}
%M_{k}=-S_{k}\left(  \mu\right)  +\underset{k\leq n\leq D_{\mathcal{T}\left(k\right)  }}{\max}S_{n}\left(  \mu\right)  .
%\end{equation}

In other words, $\Lambda_{\mathcal{T}\left(  k\right)  }$ is the time, not earlier
than $k$, at which we detect a second unsuccessful attempt at building an
upward\ patch directly. The fact that the relation in
(\ref{eq:connection M_n and stopping times}) holds, follows easily by
construction of the stopping times in (\ref{eq:construction of D_j and U_j}).
Note that it is important, however, to define $\mathcal{T}\left(  k\right)
\geq N_{0}\left(  k\right)  +1$ so that $\Lambda_{N_{0}\left(  k\right)  +1}$
is computed first. In that way, we can make sure that the maximum of the sequence
$\left(  S\left(  n\right)  :\,n\geq k\right)  $ is achieved between $k$ and
$\Lambda_{\mathcal{T}\left(  k\right)  }$.
%
%%(see Figure <ref>fig:Construction of stopping times</ref>).

\medskip{}

These observation gives rise to our suggested high-level scheme. The procedure
sequentially constructs the random walk in the intervals $\left[
\Lambda_{n-1},\Lambda_{n}\right)  $ for $n\geq1$. Here is the\textbf{\ }%
high-level procedure to construct $\left(  S\left(  n\right)  ,\,M\left(
n\right)  :n\geq0\right)  $:

%\bigskip
%%

%%%%%%%%%%%%%%%%%%%%%%%%%%%%%%only above figure.. not tbp below

%\begin{figure}[tbp]
%\begin{minipage}[t]{0.45\columnwidth}%
%\includegraphics[scale=0.1]{Hueristics_downwards_path.pdf}
%$a.$ Step 1%
%\end{minipage}%
%\begin{minipage}[t]{0.45\columnwidth}%
%\includegraphics[scale=0.1]{Hueristics_upward_path.pdf}
%$b.$ Step 2%
%\end{minipage}
%\end{figure}

\begin{algorithm}
\label{Algo Hueristics for Sn Mn} At the $k$-th iteration, for $k\geq1$: \newline%
\textbf{{Step 1: \textquotedblleft downward patch". } } Conditional on the
path not crossing $C_{UB}$ we simulate the path until we detect $\Lambda
_{k}\,$, which is  the first time when the random walk visits the set $(-\infty,S_{1}%
(\Lambda_{k-1})-m]\times...\times(-\infty,S_{l}(\Lambda_{k-1})-m]$
\label{Hueristics Step_1}.\newline\textbf{{Step 2: \textquotedblleft upward
patch". }} Check whether or not the level $S_{i}(\Lambda_{k})+m$ is ever
crossed by any of the coordinates $i$. That is, whether $\Gamma_{k}<\infty$ or
not. If the answer is \textquotedblleft Yes\textquotedblright\ then,
conditional on the path crossing $S_{i}(\Lambda_{k})+m$ for some $i$, but
not crossing the level $\left(  C_{UB}\right)  _{i}$, we simulate the path
until we detect $\Gamma_{k}$, the first time the level $S_{i}(\Lambda_{k})+m$
for at least one of the coordinates $i$ \label{Hueristics Step_2}. Otherwise
$\left(  \Gamma_{j}=\infty\right)  $, and we can update $C_{UB}$:
$C_{UB}\leftarrow S(\Lambda_{j})+m\mathbf{e}$
\end{algorithm}

The implementation of the steps in Algorithm \ref{Algo Hueristics for Sn Mn}
will be discussed in detail in the next sections, culminating with the precise
description given in Algorithm \ref{algo last} at the end of Section
\ref{Section_Final_RW}.

\subsubsection{Sampling $M\left(  0\right)  $ jointly with $\left(  S\left(
1\right)  ,...,S\left(  \Delta\right)  \right)  $}

\label{Sec_Sampling_M0}

The goal of this section is to sample exactly from $M\left(  0\right)  $. To
this end we need to simulate the sample path up to the first $\Gamma_{j}$ such
that $\Gamma_{j}=\infty$ (recall that $\Delta$ was defined to be the
corresponding $\Lambda_{j}$). This sample path will be used in the
construction of further steps in Algorithm \ref{Algo Hueristics for Sn Mn}.
This construction is directly taken from \cite{Blanchet&Chen}.

For any positive vectors $a,b>0$. Let%

\begin{equation}%
\begin{array}
[c]{lc}%
\tau_{b}=\inf\left\{  n\geq0:\,S_{i}\left(  n\right)  >b_{i}\text{ for some
}i\right\}  , & \\
\tau_{-b}=\inf\left\{  n\geq0:\,S_{i}\left(  n\right)  <-b_{i}\text{ for all
}i\right\}  , & \\
P_{a}\left(  \cdot\right)  =P\left(  \cdot\,\mid S\left(  0\right)  =a\right)
. &
\end{array}
\label{eq: Crossig times T_m T_(-m)}%
\end{equation}
Since we concentrate on $M\left(  0\right)  $, we have that $C_{UB}=\infty$.
We first need to explain a procedure to generate a Bernoulli random variable
with success parameter $P_{0}\left(  \tau_{m\mathbf{e}}<\infty\right)  $, for
suitably chosen $m>0$. Also, this procedure, as we shall see, will allow us to
simultaneously simulate $\left(  S\left(  1\right)  ,...,S\left(  \tau_{m\mathbf{e}}\right)
\right)  $ given that $\tau_{m\mathbf{e}}<\infty$.

%\subsubsection{Algorithm for sampling $Ber\left(  P_{0}\left(  \tau
%_{m\mathbf{e}}<\infty\right)  \right)  $ jointly with $\left(  S\left(
%1\right)  ,...,S\left(  \tau_{m\mathbf{e}}\right)  \right)  $ given
%$\tau_{m\mathbf{e}}<\infty$\label{Subsection_Sampling_Bernoulli}}

We think of the probability measure $P_{0}\left(  \cdot\right)  $ as defined on
the canonical space $\Omega=\{0\}\times\mathbb{R}^{l}\times\mathbb{R}%
^{l}\times...$ endowed with $\sigma$-field generated by the Borel $\sigma
$-field of finite dimensional projections (i.e. the Kolmogorov $\sigma
$-field). Our goal is to simulate from the conditional law of $(S\left(
n\right)  :0\leq n\leq\tau_{m\mathbf{e}})$ given that $\tau_{m\mathbf{e}%
}<\infty$ and $S(0)=0$, which we shall denote by $P_{0}^{\ast}$ in the rest of
this part.

First, we select $m>0$ such that
\begin{equation}
\sum_{k=1}^{l}\exp\left(  -\theta_{i}^{\ast}m\right)  <1.
\label{Eq:Assumption_m}%
\end{equation}

Now let us introduce our proposal distribution $P_{0}^{\prime}\left(
\cdot\right)  $, defined on the space $\Omega^{\prime}=\Omega\times
\{1,2,...,l\}$. We endow the probability space with the associated Kolmogorov
$\sigma$-field. So, a typical element $\omega^{\prime}$ sampled under
$P_{0}^{\prime}\left(  \cdot\right)  $ is of the form $\omega^{\prime}%
$=(($S\left(  n\right)  :n\geq0$),$Index)$, where $Index\in\{1,2,...,l\}$. The
distribution of $\omega^{\prime}$ induced by $P_{0}^{\prime}\left(
\cdot\right)  $ is described as follows, first,%
\begin{equation}
P_{0}^{\prime}\left(  Index=i\right)  =w_{i}:=\frac{\exp\left(  -\theta
_{i}^{\ast}m\right)  }{\sum_{j=1}^{l}\exp\left(  -\theta_{j}^{\ast}m\right)
}. \label{DisK}%
\end{equation}
Now, given $Index=i$, for every set $A\in\sigma$($S\left(  k\right)  :0\leq
k\leq n)$,%
\[
P_{0}^{\prime}\left(  A|Index=i\right)  =E_{0}[\exp\left(  \theta_{i}^{\ast
}Z_{i}\left(  t\right)  \right)  I_{A}].
\]
In particular, the Radon-Nikodym derivative (i.e. the likelihood ratio)
between the distribution of $\omega=(S\left(  k\right)  :0\leq k\leq n)$ under
$P_{0}^{\prime}\left(  \cdot\right)  $ and $P_{0}\left(  \cdot\right)  $ is
given by%
\[
\frac{dP_{0}^{\prime}}{dP_{0}}\left(  \omega\right)  =\sum_{i=1}^{l}w_{i}%
\exp\left(  \theta_{i}^{\ast}S_{i}\left(  n\right)  \right)  .
\]

\textit{The distribution of $(S\left(  k\right)  :k\geq0)$ under
$P_{0}^{\prime}\left(  \cdot\right)  $ is precisely the proposal distribution
that we shall use to apply acceptance / rejection. }It is straightforward to
simulate under $P_{0}^{\prime}\left(  \cdot\right)  $. First, sample $Index$
according to the distribution (\ref{DisK}). Then, conditional on $Index=i$,
the process $S\left(  \cdot\right)  $ is also a multidimensional random walk.
Indeed, given $Index=i$, under $P_{0}^{\prime}\left(  \cdot\right)  $ it
follows that $S\left(  n\right)  $ can be represented as%
\begin{equation}
S\left(  n\right)  =W^{\prime}\left(  1\right)  +...+W^{\prime}\left(
n\right)  , \label{J_prime}%
\end{equation}
where $W^{\prime}\left(  k\right)  $'s are i.i.d. with distribution obtained
by exponential titling, such that for all $A\in\sigma(W^{\prime}\left(
k\right)  )$,%
\begin{equation}
P_{0}^{\prime}(W^{\prime}\left(  k\right)  \in A)=E[\exp(\theta_{i}^{\ast
}W_{i})I_{A}]. \label{J_P2}%
\end{equation}

Now, note that we can write
\begin{align*}
E_{0}^{\prime}\left(  S_{Index}\left(  n\right)  \right)   &  =\sum_{i=1}%
^{l}E_{0}(S_{i}(n)\exp\left(  \theta_{i}^{\ast}S_{i}\left(  n\right)  \right)
)P_{0}^{\prime}\left(  Index=i\right) \\
&  =\sum_{i=1}^{l}\frac{d\psi_{i}\left(  \theta_{i}^{\ast}\right)  }{d\theta
}w_{i}>0,
\end{align*}
where the last inequality follows by convexity of $\psi_{k}\left(
\cdot\right)  $ and by definition of $\theta_{k}^{\ast}$. So, we have that
$S_{Index}\left(  n\right)  \nearrow\infty$ as $n\nearrow\infty$ with
probability one under $P_{0}^{\prime}\left(  \cdot\right)  $, by the Law of
Large Numbers. Consequently $\tau_{m\mathbf{e}}<\infty$ a.s. under
$P_{0}^{\prime}\left(  \cdot\right)  $.

Recall that $P_{0}^{\ast}\left(  \cdot\right)  $ is the conditional law of
$\left(  S\left(  n\right)  :0\leq n\leq\tau_{m\mathbf{e}}\right)  $ given
that $\tau_{m\mathbf{e}}<\infty$ and $S\left(  0\right)  =0$. In order to
assure that we can indeed apply acceptance / rejection theory to simulate from
$P_{0}^{\ast}(\cdot)$, we need to show that the likelihood ratio
$dP_{0}/dP_{0}^{\prime}$ is bounded. Indeed,
\begin{align}
  \frac{dP_{0}^{\ast}}{dP_{0}^{\prime}}\left(  S\left(  n\right)  :0\leq
t\leq\tau_{m\mathbf{e}}\right) \nonumber
&  =\frac{1}{P_{0}\left(  \tau_{m\mathbf{e}}<\infty\right)  }\times
\frac{dP_{0}}{dP_{0}^{\prime}}\left(  S\left(  n\right)  :0\leq t\leq
\tau_{m\mathbf{e}}\right) \nonumber\\
&  =\frac{1}{P_{0}\left(  \tau_{m\mathbf{e}}<\infty\right)  }\times\frac
{1}{\sum_{i=1}^{l}w_{i}\exp\left(  \theta_{i}^{\ast}S_{i}\left(
\tau_{m\mathbf{e}}\right)  \right)  }. \label{ARB}%
\end{align}
Upon $\tau_{m\mathbf{e}}$, there is an index $I^{\prime}$ ($I^{\prime}$ may be
different from $Index$) such that $\exp\left(  \theta_{I^{\prime}}^{\ast
}S_{I^{\prime}}\left(  \tau_{m\mathbf{e}}\right)  \right)  \geq\exp\left(
\theta_{I^{\prime}}^{\ast}m\right)  $, therefore%
\begin{equation}
\frac{1}{\sum_{i=1}^{l}w_{i}\exp\left(  \theta_{i}^{\ast}S_{i}\left(
\tau_{m\mathbf{e}}\right)  \right)  }\leq\frac{1}{w_{I^{\prime}}\exp\left(
\theta_{I^{\prime}}^{\ast}m\right)  }=\sum_{i=1}^{l}\exp\left(  -\theta
_{i}^{\ast}m\right)  <1, \label{B_1}%
\end{equation}
where the last inequality follows by (\ref{Eq:Assumption_m}). Consequently,
plugging (\ref{B_1}) into (\ref{ARB}) we obtain that%
\begin{equation}
\frac{dP_{0}^{\ast}}{dP_{0}^{\prime}}\left(  S\left(  n\right)  :0\leq
n\leq\tau_{m\mathbf{e}}\right)  \leq\frac{1}{P_{0}\left(  \tau_{m\mathbf{e}%
}<\infty\right)  }. \label{ARB_1}%
\end{equation}

Now we are ready to fully discuss our algorithm to sample $J$ and
$\omega=\left(  S\left(  1\right)  ,...S\left(  \tau_{m\mathbf{e}}\right)
\right)  $ given $\tau_{m\mathbf{e}}<\infty$. Upon termination we will output
the pair $\left(  J,\omega\right)  $. If $J=1$, then we set $\omega=(S\left(
1\right)  ,...,S(\tau_{m\mathbf{e}}))$. Otherwise ($J=0$), we set
$\omega=\left[  \,\right]  $, the empty vector.

\begin{algorithm}
\label{algo Sampling M_0}

INPUT:{ }${\theta_{i}^{\ast}}${ and }$m$ satisfying (\ref{Eq:Assumption_m}){.}

OUTPUT: { $J\sim Ber\left(  P_{0}\left(  \tau_{m\mathbf{e}}<\infty\right)
\right)  $ and $\omega$. If $J=1$, then $\omega=\left(  S\left(  1\right)
,\ldots,S\left(  \tau_{m\mathbf{e}}\right)  \right)  $ . Otherwise ($J=0$),
$\omega=\left[  \,\right]  $ }

Step 1: Sample $\left(  S\left(  n\right)  :0\leq t\leq\tau_{m\mathbf{e}}\right)  $
according to $P_{0}^{\prime}\left(  \cdot\right)  $ as indicated via equations
(\ref{J_prime}) and (\ref{J_P2}).

Step 2:Given $\left(  S\left(  n\right)  :0\leq t\leq\tau_{m\mathbf{e}}\right)  $,
simulate a Bernoulli $J$ with probability
\[
\frac{1}{\sum_{i=1}^{l}w_{i}\exp\left(  \theta_{i}^{\ast}S_{i}\left(
\tau_{m\mathbf{e}}\right)  \right)  }.
\]

Step 3: If $J=1$, \textbf{output} {$\left(  J,\omega\right)  $, where $\omega=\left(
S\left(  j\right)  :\,1\leq j\leq\tau_{m\mathbf{e}}\right)  $}. ELSE, if
$J=0$, \textbf{output} {$\left(  J,\omega\right)  $, where $\omega=\left[
\,\right]  $.}
\end{algorithm}

The authors in \cite{Blanchet&Chen} show that the output of the previous
procedure indeed follows the distribution of $\left(  S\left(  n\right)
:0\leq n\leq\tau_{m\mathbf{e}}\right)  $ given that $\tau_{m\mathbf{e}}%
<\infty$ and $S\left(  0\right)  =0$. Moreover, the Bernoulli random variable
$J$ has probability $P_{0}\left(  \tau_{m\mathbf{e}}<\infty\right)  $ of success.

Now we are ready to give the algorithm sampling $M(0)$ jointly with
$(S(1),...,S(\Delta))$. Before we move on to the algorithm let us define the
following. Given a vector $\mathbf{s}$, of dimension $d\geq1$, we let
$\mathbf{L}(\mathbf{s})=\mathbf{s}\left(  d\right)  $ (i.e. the $d$-th
component of the vector $\mathbf{s}$).

\begin{algorithm}
\label{algo Sampling patches of path}

INPUT {Same as Algorithm \ref{algo Sampling M_0}}

OUTPUT { The path $\left(  S\left(  1\right)  ,....,S\left(  \Delta\right)
\right)  $}\newline Initialization\ $\mathbf{s}\leftarrow\lbrack]$,
$F\leftarrow0$, and $\mathbf{L}=0$.\newline

(Initially $\mathbf{s}$ is the empty array, the variable $\mathbf{L}$
represents the last position of the drifted random walk.

WHILE {$F=0$} {\newline}

\qquad{Sample $\left(  S\left(  1\right)  ,\ldots,S\left(  \tau_{-2m\mathbf{e}%
}\right)  \right)  $ given $S\left(  0\right)  =0,$\newline$\qquad
\mathbf{s}=\left[  \mathbf{s},\mathbf{L}+S\left(  1\right)  ,\ldots
,\mathbf{L}+S\left(  \tau_{-2m\mathbf{e}}\right)  \right]  ,$\newline\qquad
}${\mathbf{L}}=\mathbf{L}+S\left(  \tau_{-2m\mathbf{e}}\right)  .$

\qquad Call Algorithm \ref{algo Sampling M_0} and obtain $\left(
J,\omega\right)  $,

\qquad IF { $J=1$ } {Set $\mathbf{s=[s,}\mathbf{L}+\omega]$,}

\qquad ELSE { $F\leftarrow1\,\,\left(  J=0\right)  $}

END WHILE

\textbf{OUTPUT }{ }\textbf{$\mathbf{s}$. }
\end{algorithm}

\begin{prop}
\label{Prop_Mo_and_upto_Delta}The output of Algorithm
\ref{algo Sampling patches of path} has the correct distribution according to
(\ref{eq def Delta}) and (\ref{Eq_EVAL_M0}). Moreover, if $\bar{N}$ is the
number of random variables needed to terminate Algorithm
\ref{algo Sampling patches of path}, there is $\delta>0$ such that
$E[\exp\left(  \delta\bar{N}\right)]  <\infty$.
\end{prop}

\textit{Proof of Proposition }\ref{Prop_Mo_and_upto_Delta}: As noted earlier,
this follows directly from the analysis in \cite{Blanchet&Chen}.\hfill$\Box$

%% % % % % % % % % % % % % % % % % % % % % % % % % %

%\subsubsection{From $M\left(  0\right)  $ to $\left(  S\left(  k\right)
%,M\left(  k\right)  \,:k\geq0\right)  $: Implementation of Procedure
%\ref{Algo Hueristics for S_n M_n}}

%\label{sec:Implementation-of-the algorithm}

%In this section we will explain in detail how to implement the steps behind
%the construction of the sequence $\left(  S\left(  n\right)  ,M\left(
%n\right)  :\,n\geq0\right)  $ that were described in Procedure
%\ref{Algo Hueristics for S_n M_n}. We will be calling Algorithm
%\ref{algo Sampling M_0} and Algorithm \ref{algo Sampling patches of path} repeatedly.

%%%%%%%%%%%%%%%%%%%%%%%%%%%%%%%%%%%%%%%%

\subsubsection{ From $M\left(  0\right)  $ to $\left(  S\left(  k\right)
,M\left(  k\right)  \,:0\leq k\leq n\right)  $\label{Section_Final_RW}}

In this section we will explain in detail the complete procedure to sample
$M(k)$ jointly with $S(k)$ for $1\leq k\leq n$, where $n$ is a finite number
given by the user. The algorithm is similar as that for sampling $(M(0)$ and
$S(1),..,S(\Delta)$ and is also based on simulating the downward and upward
patches. The main difference is that $C_{UB}<\infty$ for $M(k)$ with $k>0$ and
hence we need to simulate the random walk $S(k)$ conditional on that it never
crosses the level $C_{UB}$. In particular, we shall use the algorithm for
sampling $M(0)$ developed in Section \ref{Sec_Sampling_M0} to help us simulate
the conditional probability.

In Step \ref{Hueristics Step_1} we need to sample the maximum of the drifted
random walk $\left(  S\left(  n\right)  :\,n\geq0\right)  $. Suppose that our
current position is $S\left(  \Lambda_{j}\right)  $ and we know that the
random walk will never reach position $C_{UB}$. In other words, there exist
some $n\leq j-1$ such that $\Gamma_{n}=\infty$. Let $i=\max\{1\leq n\leq
j-1:\Gamma_{i}=\infty\}$, then $C_{UB}=S(\Lambda_{i})+m$. We now explain how
to simulate the path up to the first time $\Lambda_{\bar{n}}$, for $\bar{n}%
>j$, such that $\Gamma_{\bar{n}}=\infty$.

First, we call Algorithm \ref{algo Sampling patches of path} and obtain the
output $\omega=\left(  s_{1},...,s_{\Delta}\right)  $. We compute $M\left(
0\right)  $ according to (\ref{Eq_EVAL_M0}) and keep calling Algorithm
\ref{algo Sampling patches of path} until we obtain $M\left(  0\right)  \leq
C_{UB}-S(\Lambda_{j})$, at which point we set
%$B=I(\max\{M_{i}\left(  0\right)  :1\leq
%i\leq l\}>m)$ and write%
\begin{equation}
\left(  S\left(  \Lambda_{j}\right)  ,S\left(  \Lambda_{j}+1\right)
,\ldots,S\left(  \Lambda_{\bar{n}}\right)  \right)  =(S\left(  \Lambda
_{j}\right)  ,S\left(  \Lambda_{j}\right)  +s_{1},...,S\left(  \Lambda
_{j}\right)  +s_{\Delta}). \label{Outp}%
\end{equation}
It is clear from the construction of the path that indeed $\omega=\left(
s_{1},...,s_{\Delta}\right)  $ has the correct distribution of $\left(
S\left(  1\right)  ,...,S\left(  \Delta\right)  \right)  $ given $\tau
_{C_{UB}-S(\Lambda_{j})}=\infty$ and $S\left(  0\right)  =0$. Then, we simply
update $C_{UB}\leftarrow S\left(  \Lambda_{j}\right)  +s_{\Delta}+m\mathbf{e}$.

We close this section by giving the explicit implementation of our general
method outlined in Subsections \ref{sec:Construction S_n M_n}. In order to
describe the procedure, let us recall some definitions. Given an array
$\mathbf{s}$ of dimensions $l\times n\geq1$, let $\mathbf{L}\left(
\mathbf{s}\right)  =\mathbf{s}\left(  n\right)  $ (the last column vector of
dimension $l$ in the array). Given an array $\mathbf{z}$ of size $l^{\prime
}\times n$, set $\mathbf{d}\left(  \mathbf{z}\right)  =n$ (the number of
columns in the array). We shall evaluate $\mathbf{d}\left(  \mathbf{\cdot
}\right)  $ on arrays that might have different numbers of rows.

\begin{algorithm}
\label{algo last}

INPUT {Same as Algorithm \ref{algo Sampling M_0}}

OUTPUT {$\left(  S\left(  k\right)  ,M\left(  k\right)  \,:0\leq k\leq
n\right)  $}

Initialization $\mathbf{s}\longleftarrow\lbrack0]$, $C_{UB}\longleftarrow
\infty$, $\mathbf{N}\longleftarrow\lbrack]$, $F\longleftarrow0$. ({Initialize
the sample path with the array containing only one vector of }${l}%
${-dimensions.)}

Comments: The vector $N$, which is initially empty records the times
$\Lambda_{j}$ such that $\Gamma_{j}=\infty$. $F$ is a Boolean variable which
detects when we have enough information to compute $M\left(  n\right)  $

WHILE $F=0\newline\qquad F_{1}\longleftarrow0$\newline\qquad WHILE {$F_{1}=0$%
}\newline\qquad\qquad Call{ Algorithm \ref{algo Sampling patches of path}.
Obtain as output $\omega=(s_{1},...,s_{\Delta})$, and get }$M\left(  0\right)
.$\newline\qquad\qquad{IF }$M\left(  0\right)  \leq C_{UB}-\mathbf{L}%
(\mathbf{s})$, update $C_{UB}=\mathbf{L}(\mathbf{s})+s_{\Delta}+M(0)\mathbf{e}%
$, $\mathbf{s=[s,}\mathbf{L}(\mathbf{s})+\omega]$, $\mathbf{N}=[\mathbf{N}%
,\mathbf{d}\left(  \mathbf{s}\right)  ]$ and $F_{1}=1$.\newline\qquad END
WHILE\newline
%	Set $C_{UB}\longleftarrow C_{UB}-\mathbf{L}(\mathbf{s})$

\qquad IF {$\mathbf{N}\left(  \mathbf{d(N)}-1\right)  \geq n$}, {set
$F\leftarrow1$.}\newline END WHILE$\newline$FOR {$k=0,...,n$} \newline%
\qquad{$M\left(  k\right)  =\max(\mathbf{s}\left(  k+1\right)  ,\mathbf{s}%
\left(  k+2\right)  ,....,\mathbf{s}\left(  \mathbf{d(s)}\right)  )$,}%
\newline\qquad{$S\left(  k\right)  =\mathbf{s(}k+1).$}\newline END
FOR\newline\medskip\textbf{OUTPUT:} $\left(  S\left(  k\right)  ,M\left(
k\right)  \,:1\leq k\leq n\right)  $.
\end{algorithm}

%\bibliographystyle{plainnat}
%\bibliography{reference}

\section{Numerical Results}

To test the numerical performance and correctness of our algorithm, we
implement our algorithm in Matlab. In particular, we consider a 2-station
Jackson network with Poisson arrivals and exponential service times, so that
the true value of the steady-state distribution is known in closed form. In
the numerical test, we shall fix the routing matrix $Q=[0, 0.11; 0.1, 0]$ and
run the simulation algorithm for different arrival and service rates $\lambda$
and $\mu$. For each pair of $(\lambda, \mu)$, we generate 10000 i.i.d. samples
of the number of customers $(Y_{1}(\infty), Y_{2}(\infty))$.

We estimate the steady-state expectation $E[Y_{i}(\infty)]$ and the
correlation coefficient of $Y_{1}(\infty)$ and $Y_{2}(\infty)$ based on the
10000 i.i.d. samples. Since the 2-station system is a Jackson network, the
theoretic steady-state distribution of $Y_{i}(\infty)$ is known and the true
value of $E[Y_{i}(\infty)]=\phi/(\mu-\phi)$. Moreover, the true value of the
correlation coefficient is 0 as the joint distribution of $(Y_{1}(\infty),
Y_{2}(\infty))$ is of product form. In Table 1, for different $\mu$ and
$\lambda$, we report the simulation estimations and compare them with and the
true values. In detail, we report the 95\% confidence interval of
$E[Y_{i}(\infty)]$ estimated from the simulated samples. For the correlation,
we report the sample correlation coefficient and the $p$-value of the
hypothesis test that the two population are not correlated.

%% Here insert the table

\begin{table}[h]
\caption{{\protect\footnotesize Simulation Estimation of $E[Y_{1}(\infty)]$,
$E[Y_{2}(\infty)]$ and $\text{Corr}(Y_{1}(\infty), Y_{2}(\infty))$ for
different $\lambda$ and $\mu$. ($p$-value $> 5$\% means no significant
correlation)}}%
\label{my-label}
\centering
\begin{tabular}
[c]{||l|l|l|l|l|l||}\hline
\multicolumn{6}{||l||}{Parameters}\\\hline
$\lambda$ & (0.2250, 0.7170) & (0.2200, 0.7670) & (0.2180, 0.7870) & (0.2160,
0.8070) & (0.2140, 0.8270)\\\hline
$\mu$ & (1.0000, 1.0000) & (1.0000, 1.0000) & (1.0000, 1.0000) & (1.0000,
1.0000) & (1.0000, 1.0000)\\\hline
\multicolumn{6}{||l||}{$E[Y_{1}(\infty)]$}\\\hline
TrueValue & 0.4286 & 0.4286 & 0.4286 & 0.4286 & 0.4286\\\hline
Simulation & 0.4265$\pm$0.0152 & 0.4204$\pm$0.0150 & 0.4247$\pm$0.0150 &
0.4376$\pm$0.0153 & 0.4228$\pm$0.0155\\\hline
\multicolumn{6}{||l||}{$E[Y_{2}(\infty)]$}\\\hline
TrueValue & 3.0000 & 4.0000 & 4.5556 & 5.2500 & 6.1429\\\hline
Simulation & 2.9355$\pm$0.0676 & 4.0468$\pm$0.0877 & 4.5844$\pm$0.0984 &
5.3057$\pm$0.1156 & 6.1620$\pm$0.1291\\\hline
\multicolumn{6}{||l||}{$\text{Corr}(Y_{1}(\infty),Y_{2}(\infty))$}\\\hline
Simulation & -0.0058 & -0.0128 & 0.0151 & 0.0011 & 0.0116\\\hline
$p$-value & 55.96\% & 19.90\% & 13.13\% & 91.13\% & 24.80\%\\\hline
\end{tabular}
\end{table}

Figure 1 and 2 compares the histogram of the 10000 simulation samples with the
true steady state distribution for two different values of $\lambda$ and $\mu
$. In both two cases, we can see that the empirical distribution of the i.i.d.
simulated samples is very close to the true distribution.
%% figure 1
\begin{figure}[ht]
\begin{center}
\begin{subfigure}{0.5\textwidth}
\includegraphics[scale=0.36]{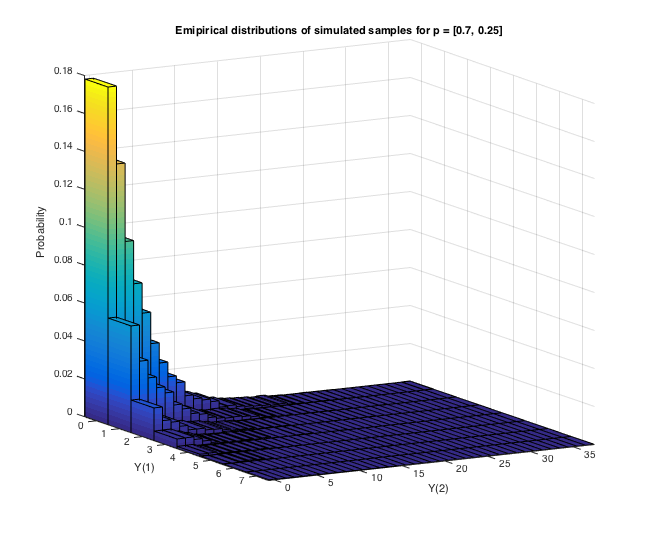}
\caption{}
\end{subfigure}
%\hspace{.5in}
\begin{subfigure}{0.5\textwidth}
\includegraphics[scale=0.36]{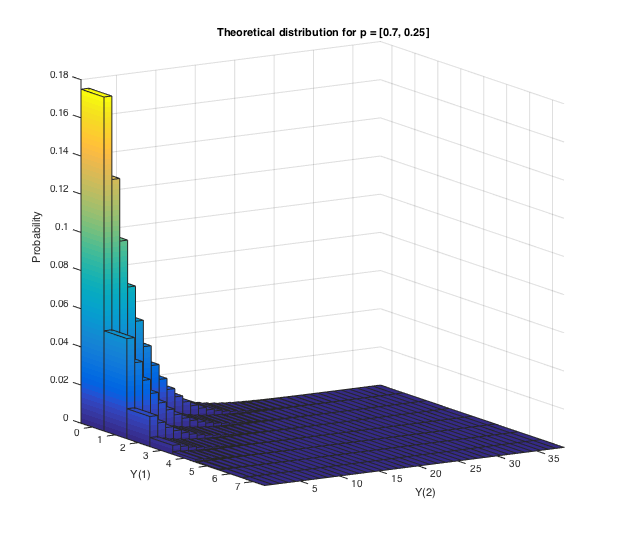}
\caption{}
\end{subfigure}
\par
\begin{subfigure}{0.5\textwidth}
\includegraphics[scale=0.36]{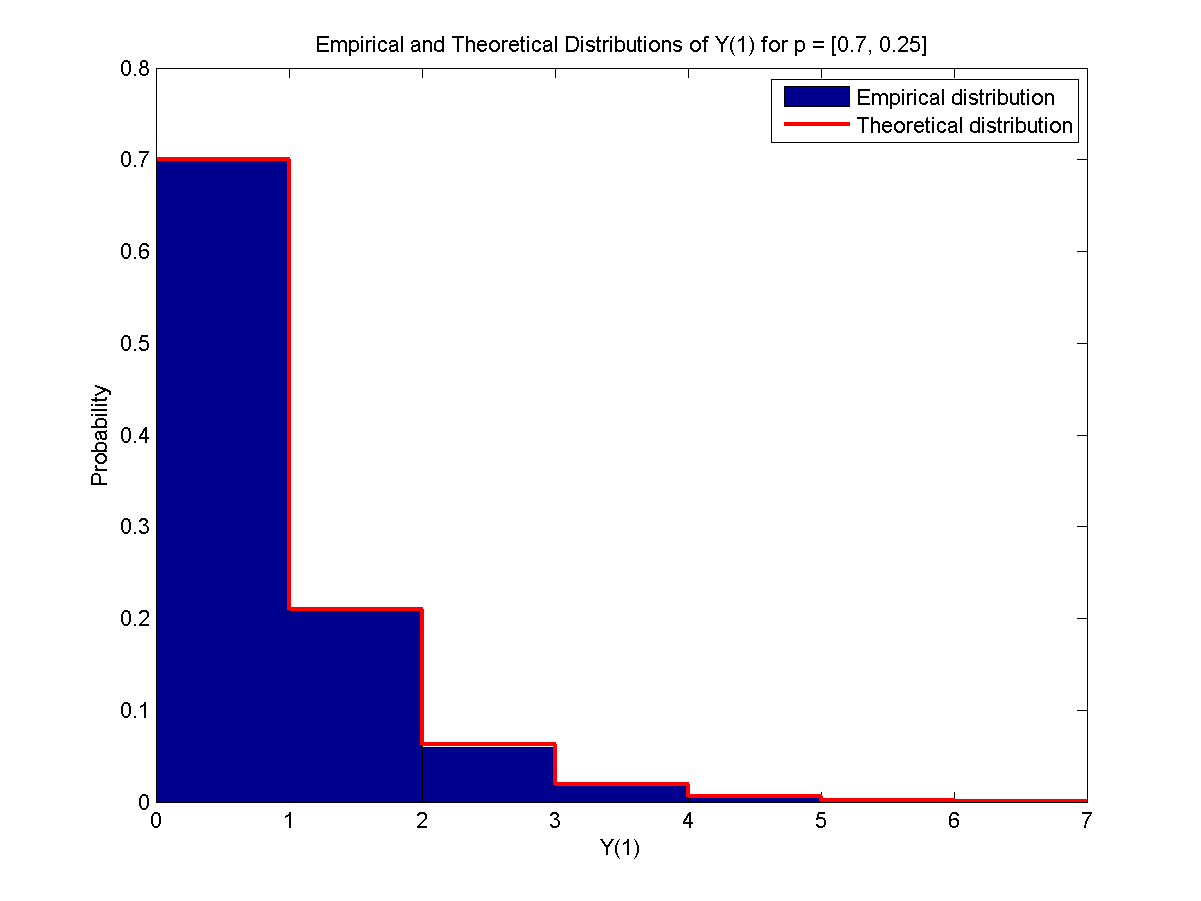}
\caption{}
\end{subfigure}\begin{subfigure}{0.5\textwidth}
\includegraphics[scale=0.36]{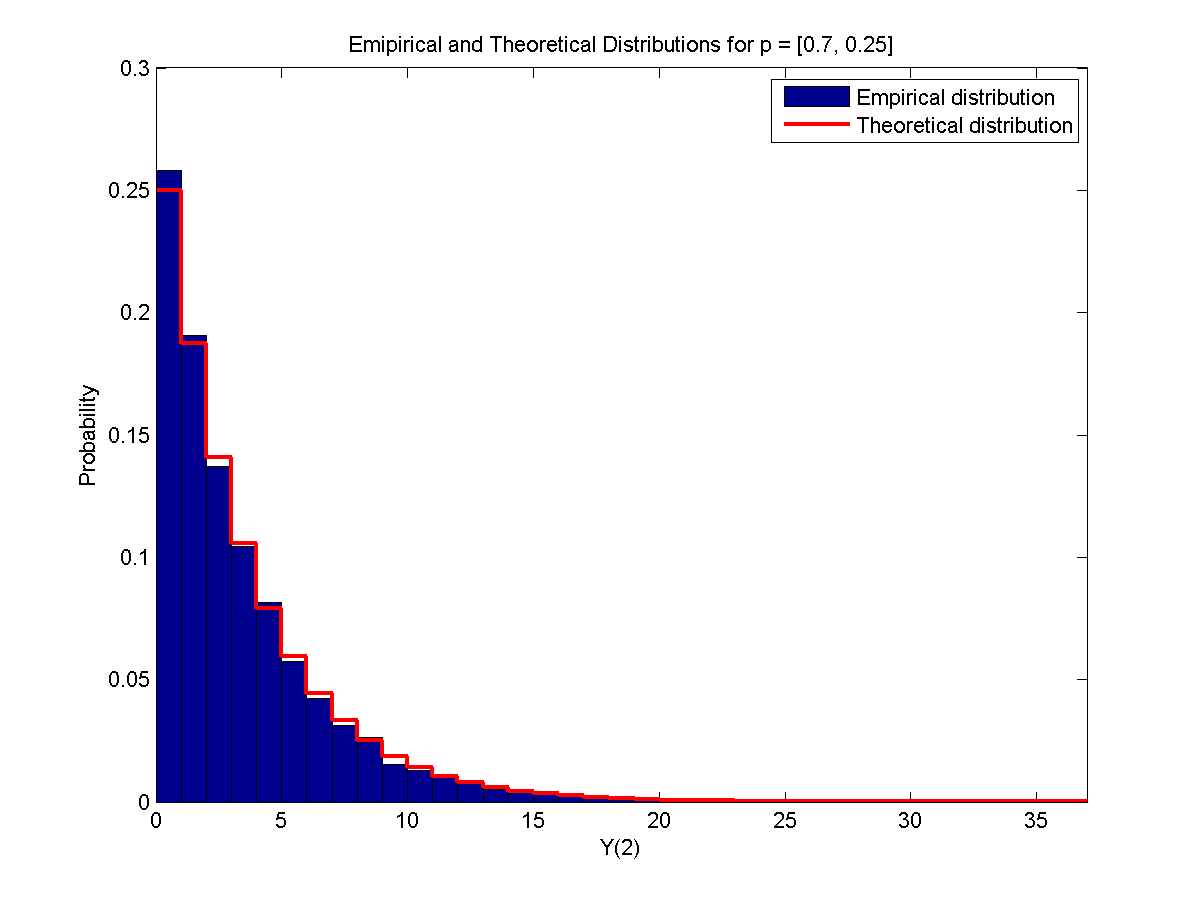}
\caption{}
\end{subfigure}
\end{center}
\caption{{\protect\footnotesize $\lambda=(0.225, 0.717)$, $\mu=(1,1)$. (A)
Histogram of the 10000 simulated samples of $(Y_{1}(\infty), Y_{2}(\infty))$.
(B) Theoretic steady-state distribution of $(Y_{1}(\infty), Y_{2}(\infty))$.
(C) Marginal distribution of $Y_{1}(\infty)$. (D) Marginal distribution of
$Y_{2}(\infty)$. }}%
\label{fig1}%
\end{figure}

%%figure 2
\begin{figure}[ht]
\begin{center}
\begin{subfigure}{0.5\textwidth}
\includegraphics[scale=0.40]{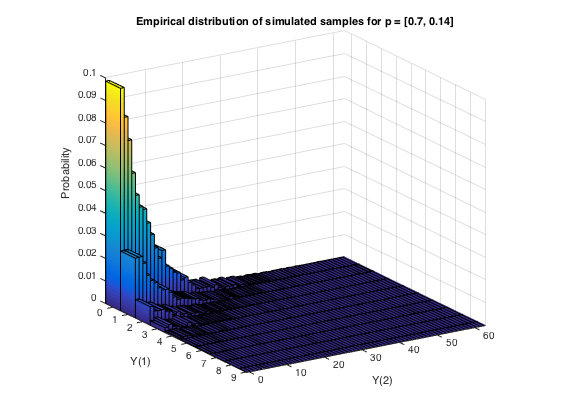}
\caption{}
\end{subfigure}
%\hspace{.5in}
\begin{subfigure}{0.5\textwidth}
\includegraphics[scale=0.40]{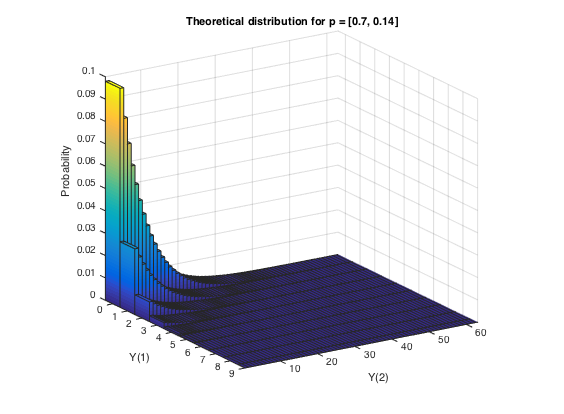}
\caption{}
\end{subfigure}
\par
\begin{subfigure}{0.5\textwidth}
\includegraphics[scale=0.36]{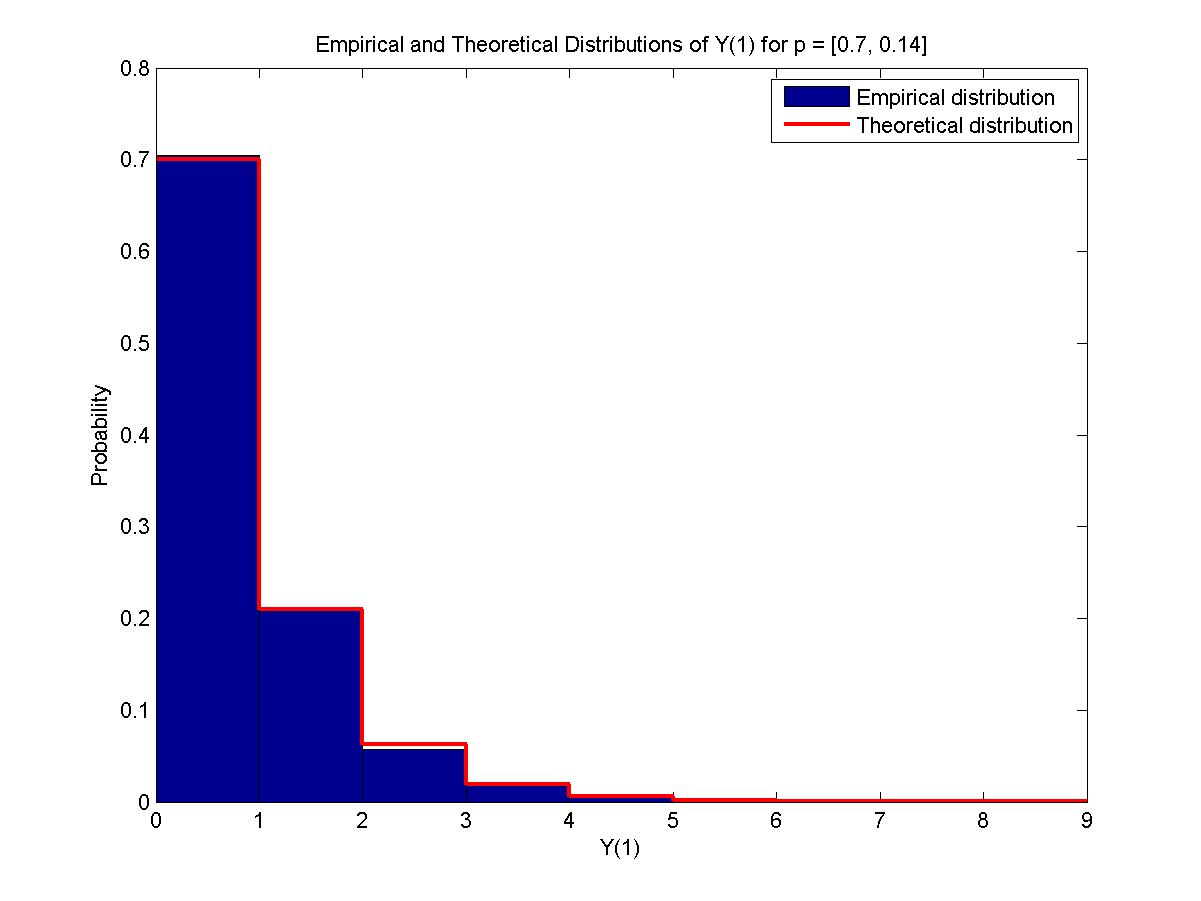}
\caption{}
\end{subfigure}\begin{subfigure}{0.5\textwidth}
\includegraphics[scale=0.36]{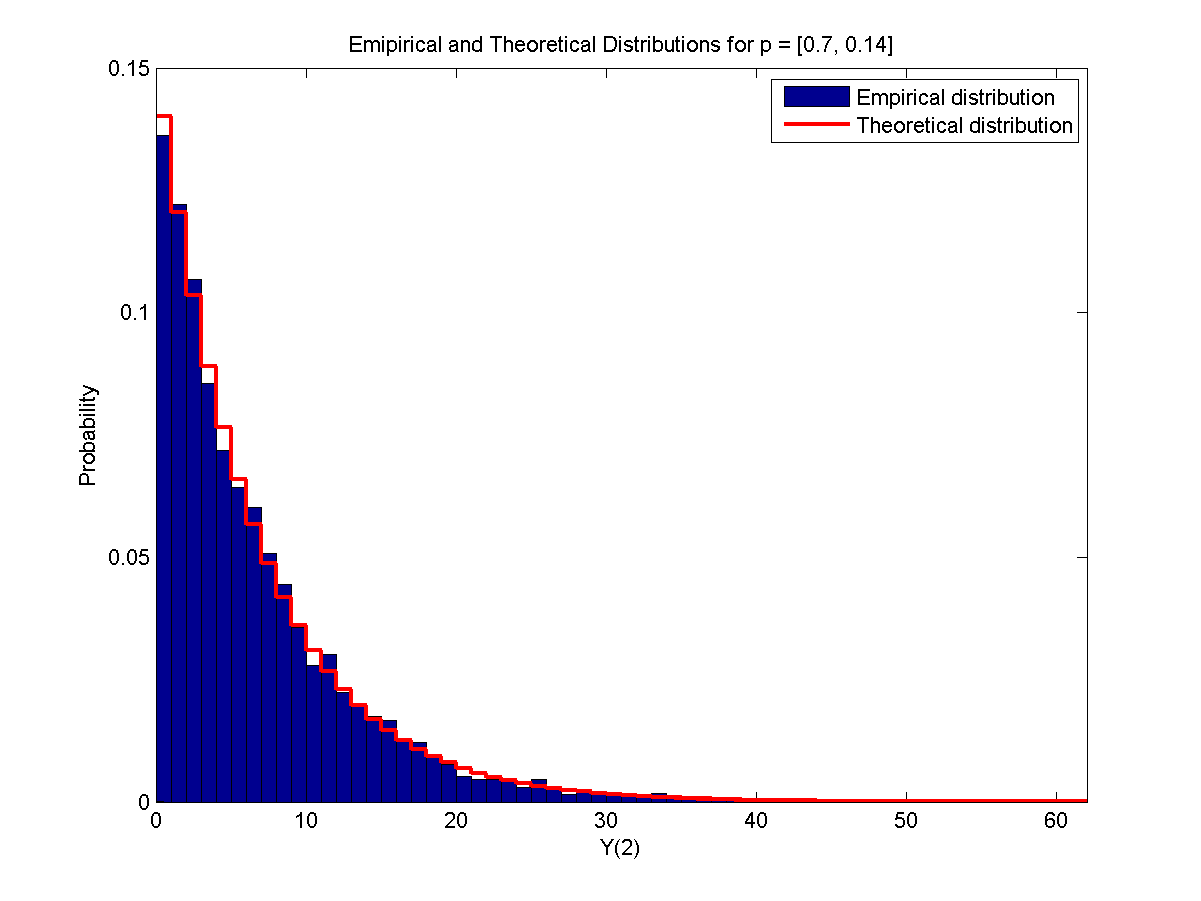}
\caption{}
\end{subfigure}
\end{center}
\caption{{\protect\footnotesize $\lambda=(0.214, 0.827)$, $\mu=(1,1)$. (A)
Histogram of the 10000 simulated samples of $(Y_{1}(\infty), Y_{2}(\infty))$.
(B) Theoretic steady-state distribution of $(Y_{1}(\infty), Y_{2}(\infty))$.
(C) Marginal distribution of $Y_{1}(\infty)$. (D) Marginal distribution of
$Y_{2}(\infty)$. }}%
\label{fig1}%
\end{figure}
\newpage
\section{Appendix: Technical Proofs}

\subsection{Technical Lemmas and the Proof of Theorem \ref{Thm_Domination} .}

Part i) of Theorem \ref{Thm_Domination} is a restatement of Lemma 4.2 in
\cite{Chang&etal}. Then, part ii)\ follows from the next lemma.

\begin{lem}
\label{lemma: vacation system initial state}Suppose we start the coupled
systems $\mathcal{N}^{+}$ and $\mathcal{N}^{\prime}$ empty from time $0$. Then
for any $t>0$, when $Y_{i}^{\prime}(t)=y_{i}$, then the service station $i$ in
system $\mathcal{N}^{+}$ must be one of the three following cases:

\begin{enumerate}
\item $Y^{+}_{i}(t)=y_{i}+1$ and the server is in service

\item $Y^{+}_{i}(t)=y\in\{1,2,...,y_{i}\}$ and the server is either in service
or vacation

\item $Y_{i}^{+}(t)=0$ and the server is in vacation.
\end{enumerate}
\end{lem}

\textit{Proof of Lemma \ref{lemma: vacation system initial state}}: The result
follows directly by comparing the evolution of $Y_{i}^{\prime}(\cdot)$ given
by (\ref{SDE}), against the evolution of the number of customers waiting in
the $i$-th queue of $\mathcal{N}^{+}$, namely $\widehat{Y}_{i}^{+}\left(
\cdot\right)  $, which satisfies (\ref{Eq:Evolution_Vacation}). The equations
are monotone with respect to the input process, which is strictly smaller for
the network $\mathcal{N}^{+}$ compared to $\mathcal{N}^{\prime}$ because
\[
\sum_{j:j\neq i,1\leq j\leq d}I(\widehat{S}_{j}^{+}\left(  t_{-}\right)
>0)\text{d}D_{j,i}\left(  t\right)  \leq\sum_{j:j\neq i,1\leq j\leq d}%
\text{d}D_{j,i}\left(  t\right)  .
\]
So, we conclude that $\widehat{Y}_{i}^{+}\left(  t\right)  \leq Y_{i}^{\prime
}(t)$. Therefore, if $Y_{i}^{\prime}(t)=y_{i}$, we have $\widehat{Y}_{i}%
^{+}\left(  t\right)  \leq y_{i}$. As $\widehat{Y}_{i}^{+}\left(  t\right)  $
is the number of customers who are waiting for entering service, we can
conclude that either $Y_{i}^{+}(t)=\widehat{Y}_{i}^{+}\left(  t\right)  +1$ if
the server is in service and $Y_{i}^{+}(t)=\widehat{Y}_{i}^{+}\left(
t\right)  $ if the server is on vacation, and hence we are done.\hfill$\Box$

In order to prove part iii) of Theorem \ref{Thm_Domination} we introduce some notation.

Let $y=(y_{1},...,y_{d})$ be a fixed vector in $\mathbb{N}_{0}^{d}%
=\{0,1,...\}^{d}$. Consider three vacation networks $\mathcal{N}^{+}$,
$\mathcal{N}^{++}$, $\mathcal{N}^{+-}$ that have the same network topology and
are driven by the same arrival and activity sequences, namely, $(A_{i}\left(
n\right)  :n\geq0)$ and $(\upsilon_{i}^{0}\left(  n\right)  ,r_{i}\left(
n\right)  :n\geq1)$, except for their initial state at time $0$. In
particular, we set $Y_{i}^{++}(0)=y_{i}+1$ with all servers in service for all
$i$, and $Y_{i}^{+-}(0)=0$ with all servers in vacation. The state of each
service station in $\mathcal{N}^{+}$ at time $0$ is of any one of three cases
as described in Lemma \ref{lemma: vacation system initial state}. More
precisely, we have the following system of SDEs for $\widehat{Y}_{i}^{++}$,
$\widehat{S}_{i}^{++}$, and $Y_{i}^{++}$ with $i\in\{1,...,d\}$,%
\begin{align*}
\text{d}\widehat{Y}_{i}^{++}\left(  t\right)   &  =\text{d}N_{0,i}\left(
t\right)  +\sum_{j:j\neq i,1\leq j\leq d}I(\widehat{S}_{j}^{++}\left(
t_{-}\right)  >0)\text{d}D_{j,i}\left(  t\right)  -I(\widehat{Y}_{i}%
^{++}\left(  t_{-}\right)  >0)\text{d}D_{i}\left(  t\right)  ,\\
\text{d}\widehat{S}_{i}^{++}\left(  t\right)   &  =(I(\widehat{Y}_{i}%
^{++}\left(  t_{-}\right)  >0)-I(\widehat{S}_{i}^{++}\left(  t_{-}\right)
>0))\text{d}D_{i}\left(  t\right)  ,\\
\widehat{Y}_{i}^{++}\left(  0\right)   &  =y_{i}+1,\text{ }\widehat{S}%
_{i}^{++}\left(  0\right)  =1,\\
Y_{i}^{++}\left(  t\right)   &  =\widehat{Y}_{i}^{++}\left(  t\right)
+\widehat{S}_{i}^{++}\left(  t\right)  .
\end{align*}
The SDEs for $\widehat{Y}_{i}^{+},\widehat{S}_{i}^{+}$, $Y_{i}^{+}$, and
$\widehat{Y}_{i}^{+-},\widehat{S}_{i}^{+-}$, $Y_{i}^{+-}$are exactly the same,
except for the boundary conditions. In particular, $\widehat{Y}_{i}^{+}\left(
0\right)  =y_{i},$ $\widehat{S}_{i}^{++}\left(  0\right)  \in\{0,1\}$, and
$\widehat{Y}_{i}^{+-}\left(  0\right)  =0,$ $\widehat{S}_{i}^{+-}\left(
0\right)  =0$. Then we have the following comparison result which implies part
iii) of Theorem \ref{Thm_Domination}:

\begin{lem}
\label{Lemma_Comp_Vacation_vs_Vacation}In order to distinguish servers
whenever there might be ambiguity we shall call the server of the $i$-th
service station in $\mathcal{N}^{++}$ the server $i^{+}$, the server of the
$i$-th station in $\mathcal{N}^{+}$ is called server $i$, and the $i$-th
server in $\mathcal{N}^{+-}$ is called $i^{-}$. We claim that the following
three statements hold for all servers $i$, $i^{+}$, and $i^{-}$
$(i=1,2,...,d)$ and at any $t\geq0$ (analogous statements to 2. and 3. hold
replacing $i^{+}$ by $i$ and $i$ by $i^{-}$)

\begin{enumerate}
\item $Y_{i}^{++}(t)\geq Y_{i}^{+}(t)\geq Y_{i}^{+-}(t)$.

\item If $Y_{i}^{++}(t)=Y_{i}^{+}(t)$, server $i^{+}$ and server $i$ are both
in service or both in vacation. Similarly, $Y_{i}^{+}(t)=Y_{i}^{+-}(t)$,
server $i$ and server $i^{-}$ are both in service or both in vacation

\item If server $i^{+}$ is in vacation, server $i$ is also in vacation.
Similarly, if $i$ is in vacation, server $i^{-}$ is also in vacation.
\end{enumerate}
\end{lem}

\textit{Proof of Lemma \ref{Lemma_Comp_Vacation_vs_Vacation}}: Let's first
prove $Y_{i}^{++}(t)\geq Y_{i}^{+}(t)$.

First let's check if the claim is true for $t=0$. Note that $Y_{i}%
^{++}(0)=y_{i}+1$ and $Y_{i}^{+}\in\{0,1,...,y_{i}+1\}$ and hence Statement
(1) holds. As server $i^{+}$ is in service at time 0, Statement (2) also hold.
Finally, if $Y_{i}^{+}(0)=Y_{i}^{++}(0)=y_{i}+1$, service station $i$ is of
the first case in Lemma \ref{lemma: vacation system initial state} which means
both server $i$ and $i^{+}$ are at service. In summary, the claim is true for
$t=0$.\newline

Let $E(t)=\sum_{i=1}^{d}N_{0,i}(t)+\sum_{i=1}^{d}D_{i}(t)$ be the counting
process of events that occur in the network. Define $t(n)=\inf\{t\geq
0:E(t)=n\}$ to be the time at which the $n$-th event occurs for $n\geq1$ and
set $t\left(  0\right)  =0$. We shall prove the statements 1-3 only for $i$
and $i^{+}$ first by induction on $n\geq0$ at times $t\left(  n\right)  $,
since there are no changes inside the network population between two event
epochs. We have verified that statements 1-3 are valid at $t\left(  0\right)
$. Assume by induction hypothesis, that statements 1-3 hold for $t(n-1)$, we
need to consider several cases at time $t(n)$.\newline

\textit{Case 1}: $t(n)$ corresponds to an arrival from $N_{0,i}(\cdot
)$.\newline In this case, $Y_{i}^{++}(t(n))=Y_{i}^{++}(t(n-1))+1$ and
$Y_{i}^{+}(t(n))=Y_{i}^{+}(t(n-1))+1$. According to the dynamics of vacation
system, a new arrival from $N_{0,i}(\cdot)$ does not change the type of
activity that is going on in servers $i^{+}$ and $i$. So statements 1-3 hold
for server $i$ and $i^{+}$ at $t(n)$. As to all the other servers, there are
no changes between $t(n-1)$ and $t(n)$. In summary, Statement 1-3 hold for all
servers at $t(n)$.\newline

\textit{Case 2}: $t(n)$ corresponds to an arrival from $D_{i}(\cdot)$ and
$Y_{i}^{++}(t(n-1))=Y_{i}^{+}(t(n-1))$.\newline By induction hypothesis,
server $i$ and $i^{+}$ are in the same type of activity at time $t\left(
n-1\right)  $. Suppose that both servers $i$ and $i^{+}$ are at vacation at
$t(n-1)$, it is clear from the dynamics that $Y_{i}^{++}(t(n))=Y_{i}%
^{+}(t(n))$. If $Y_{i}^{++}(t(n))>0$, it means that there was someone waiting
and therefore at time $t\left(  n\right)  $, coming from vacation, now both
$i^{+}$ and $i$ are now in service at time $t(n)$; otherwise, from the same
logic, $Y_{i}^{++}(t(n))=0$, implies that both $i^{+}$ and $i$ are on vacation
at $t(n)$. Besides, there are no changes on other servers between $t(n-1)$ and
$t(n)$, because at $t(n-1)$ the servers where on vacation. Therefore,
Statement 1-3 hold for all servers at $t(n)$.\newline

If both server $i$ and $i^{+}$ are in service at $t(n-1)$ and $Y_{i}%
^{++}(t(n-1))=Y_{i}^{+}(t(n-1))=1$, then $Y_{i}^{++}(t(n))=Y_{i}^{+}(t(n))=0$
and both server $i^{+}$ and $i$ are in vacation at $t(n)$. Let $j=r_{i}%
(B_{i}(t(n)))$. If $j=0$, there are no changes on other servers between
$t(n-1)$ and $t(n)$, so Statement 1-3 hold for all servers at $t(n)$. If
$j>1$, then we can apply the argument of \textit{Case 1} to server $j$,
$j^{+}$, and there are no changes on the rest servers other than $i^{+},i,$
$j^{+}$ and $j$. So statements 1-3 hold for all servers at $t=t(n)$. \newline

If both server $i$ and $i^{+}$ are in service at $t(n-1)$ and $Y_{i}%
^{++}(t(n-1))=Y_{i}^{+}(t(n-1))>1$, the argument is similar to when
$Y_{i}^{++}(t(n-1))=Y_{i}^{+}(t(n-1))=1$ except that now both server $i^{+}$
and $i$ are in service at $t(n)$.\newline

\textit{Case 3}: $t(n)$ corresponds to an arrival from $D_{i}(\cdot)$ and
$Y_{i}^{++}(t(n-1))>Y_{i}^{+}(t(n-1))$\newline

If server $i^{+}$ is in vacation at $t(n-1)$, by induction hypothesis, server
$i$ is also in vacation at $t(n-1)$. Then, there are no changes on all other
servers. Besides, $Y_{i}^{++}(t(n))=Y_{i}^{++}(t(n-1))$ and $Y_{i}%
^{+}(t(n))=Y_{i}^{+}(t(n-1))$ and hence $Y_{i}^{++}(t(n))>Y_{i}^{++}(t(n))$.
As $Y_{i}^{++}(t(n))>0$, server $i^{+}$ is in service at $t(n)$ and hence we
do not contradict statement 3 for servers $i^{+}$ and $i$ at time $t(n)$. In
summary, we conclude that Statements 1-3 hold for all servers at time
$t(n)$.\newline

If server $i^{+}$ is in service and $Y_{i}^{++}(t(n-1))=1$ (and so $Y_{i}%
^{+}(t(n-1))=0$), $Y_{i}^{++}(t(n))=Y_{i}^{+}(t(n))=0$ and server $i^{+}$ and
$i$ are both in vacation at time $t(n)$. Let $j=r_{i}(D_{i}(t(n)))$. If $j=0$,
there are no changes on all other servers and hence statements 1-3 hold for
all servers at $t(n)$. Otherwise, we have $Y_{j}^{++}(t(n))=Y_{j}%
^{++}(t(n-1))+1$ and $Y_{j}^{+}(t(n))=Y_{j}^{+}(t(n-1))$, as $Y_{j}%
^{++}(t(n-1))\geq Y_{j}^{+}(t(n-1))$ by induction hypothesis, $Y_{j}%
^{++}(t(n))>Y_{j}^{+}(t(n))$. So statement 1-2 hold for server $j^{+}$ and
$j$. The type of activity that occurs in server $j^{+}$ and $j$ remains the
same what was going on at time $t(n-1)$ and hence statement 3 holds. Since
there are no changes on the rest servers other than $i^{+},i,j^{+}$ and $j$,
statement 1-3 hold at time $t(n)$ for all servers.\newline

If server $i^{+}$ is in service at $t(n-1)$ and $Y_{i}^{++}(t(n-1))>1$,
$Y_{i}^{++}(t(n))=Y_{i}^{++}(t(n-1))-1>0$ and server $i^{+}$ is still in
service at $t(n)$. So statement 3 holds for servers $i^{+}$ and $i$ at time
$t(n)$. As $Y_{i}^{++}(t(n-1))\geq Y_{i}^{+}(t(n-1))+1$ and $Y_{i}%
^{+}(t(n))\geq Y_{i}^{+}(t(n))$, $Y_{i}^{++}(t(n))\geq Y_{i}^{+}(t(n))$ and
statement 1 holds for server $i$. In case $Y_{i}^{++}(t(n))=Y_{i}^{+}(t(n))$,
$Y_{i}^{+}(t(n))>0$ and hence both server $i^{+}$ and $i$ are in service at
$t(n)$ and statement 2 holds. Following a similar argument as when server
$i^{+}$ is in service at $t(n-1)$ and $Y_{i}^{+}(t(n-1))=1$, we can check that
statement 1-3 hold for all the other servers. As a result, we can conclude
that statement 1-2 hold at time $t(n)$ for all servers.\newline

By induction, and by the nature of the processes, which changes only at times
$t\left(  n\right)  $, statements 1-3 hold for all $t\geq0$.\newline

To prove that $Y_{i}^{+}(t)\geq Y_{i}^{+-}(t)$, we can use the same induction
arguments simply replacing $Y_{i}^{++}(t)$ with $Y_{i}^{+}(t)$, and $Y_{i}%
^{+}(t)$ with $Y_{i}^{+-}(t)$ in statements 1-3. The induction part is exactly
the same, so we are done if we can check that the three statements all hold at
time $t\left(  0\right)  $.\newline

As $Y_{i}^{+-}(0)=0$ and all servers $i^{-}$ are in vacation, statement 1-3
immediately hold. If $Y_{i}^{+}(0)=Y_{i}^{+-}(0)$, then $Y_{i}^{+}(0)=0$ and
service station $i$ is in the last case as in Lemma
\ref{lemma: vacation system initial state}, hence both $i$ and $i^{-}$ are in
vacation and statement 2 holds. In summary, Statement 1-3 all hold for time
$t\left(  0\right)  =0$ and thus the result follows.\hfill$\Box$

\subsection{Recapitulation of the Main Procedure and Proof of Theorem
\ref{Thm_Main}\label{Sec_Proof_Main}}

In order to prove Theorem \ref{Thm_Main}, we need to recapitulate on the
execution of our Main Procedure. Let us go back to equation
(\ref{Eq:Evolution_Vacation}) and allow us write
\[
Y_{i}^{+}\left(  t;T,y\right)  =Y_{i}^{+}\left(  t;T\right)
\]
to recognize the boundary condition in (\ref{Eq:Evolution_Vacation}).
Moreover, we recall that $y_{i}=\bar{Y}_{i}^{\prime}\left(  -T\right)  $, from
equation (\ref{Eq:InitialCondition}). For any $T>0$ define the event
\[
C_{T}=\{\text{for all }t\in\lbrack0,T]\text{ there is }i\text{ such that
}Y_{i}^{+}(k;T,\bar{Y}^{\prime}\left(  -T\right)  )>0\}.
\]
Then put $\bar{\tau}=\inf\{t\geq0:\bar{C}_{t}$ occurs$\}$. Assuming that the
output indeed follows the steady state distribution, the statement of Theorem
\ref{Thm_Main} concerning the computational cost measure in terms of random
numbers generated will follows if we can show that there exists $\delta>0$
such that $E[\exp\left(  \delta\bar{\tau}\right)]  <\infty$.

We start by noting that
\[
P\left(  \bar{\tau}>u\right)  =P\left(  C_{u}\right)  .
\]
In order to compute $P\left(  C_{n}\right)  $ we can think forward in time, in
particular consider
\begin{align*}
\text{d}\widehat{Y}_{i}^{+}\left(  t;0,y\right)   &  =\text{d}N_{0,i}\left(
t;0\right)  +\sum_{j:j\neq i,1\leq j\leq d}I(\widehat{S}_{j}^{+}\left(
t_{-};0,y\right)  >0)\text{d}D_{j,i}\left(  t;0\right)  -I(\widehat{Y}_{i}%
^{+}\left(  t_{-};0,y\right)  >0)\text{d}D_{i}\left(  t;0\right)  ,\\
\text{d}\widehat{S}_{i}^{+}\left(  t;0,y\right)   &  =(I(\widehat{Y}_{i}%
^{+}\left(  t_{-};0,y\right)  >0)-I(\widehat{S}_{i}^{+}\left(  t_{-}%
;0,y\right)  >0))\text{d}D_{i}\left(  t;0\right)  \text{,}\\
\widehat{Y}_{i}^{+}\left(  0;0,y\right)   &  =y,\text{ }\widehat{S}_{i}%
^{+}\left(  0;0,y\right)  =1.
\end{align*}
Note the relation between $D_{i}\left(  t;0\right)  $ and $D_{i}\left(
-t\right)  $, defined in (\ref{Eq:Two_Sided_Processes}), in particular
$D_{i}(-t)=-D_{i}\left(  t;0\right)  \leq0$ (similarly $N_{0,i}\left(
-t\right)  =N_{0,i}\left(  t;0\right)  $). Then let $Y_{i}^{+}\left(
t;0,y\right)  =\widehat{Y}_{i}^{+}\left(  t;0,y\right)  +\widehat{S}_{i}%
^{+}\left(  t;0,y\right)  $ we have that
\[
P\left(  C_{u}\right)  =P\left(  \text{for all }t\in\lbrack0,u]\text{, there
is }i\text{ such that }Y_{i}^{+}\left(  t;0,\bar{Y}^{\prime}\left(  0\right)
\right)  >0\right)  .
\]
The strategy is to first describe the evolution of $Y^{+}\left(
\cdot;0,y\right)  $ in terms of a Markov process. We need to track the
residual times associated with each renewal process and the number of people
both in queue and in service in each station. In particular, define
\[
G_{i}\left(  t\right)  =\sup\{|A_{i}\left(  -n\right)  |:1\leq n\leq
N_{0,i}\left(  t;0\right)  +1\}-t.
\]
Similarly, we define
\[
H_{i}\left(  t\right)  =\sup\{|B_{i}\left(  -n\right)  |:1\leq n\leq
D_{i}\left(  t;0\right)  \}-t.
\]
Then we let $t\left(  n\right)  $ be the times at which events occur, that is,
$t\left(  1\right)  <t\left(  2\right)  <...$ are the discontinuity points of
the process $\{E\left(  t\right)  :t\geq0\}$ defined as $E(t)=\sum_{i}%
N_{0,i}(t;0)+\sum_{i}D_{i}(t;0)$. Let us write $\Theta_{i}^{+}(n)=(\widehat{Y}%
_{i}^{+}(t(n)),\widehat{S}_{i}^{+}(t(n)))$ and define $\Xi^{+}\left(
n\right)  =(\Xi_{i}\left(  n\right)  :1\leq i\leq d)$ as
\[
\Xi_{i}^{+}\left(  n\right)  =(\Theta_{i}^{+}(n),G_{i}\left(  t\left(
n\right)  \right)  ,H_{i}\left(  t\left(  n\right)  \right)  ).
\]
Note that $\{\Xi^{+}\left(  n\right)  :n\geq0\}$ forms a Markov chain and we
are given the initial condition $\Xi_{i}^{+}\left(  0\right)  =(\bar{Y}%
_{i}^{\prime}(0),1,G_{i}\left(  0\right)  ,H_{i}\left(  0\right)  )$. Now,
define
\[
\tau\left(  c\right)  =\inf\{n\geq0:\sum_{i=1}^{d}\Theta_{i}^{+}\left(
n\right)  \leq c\},
\]
for some $c>0$. Following a similar approach to \cite{Gamarnik&Zeevi}, due to
Assumption 2, we now can show that there exists $c>0$ such that $E[\exp
(\delta\tau\left(  c\right)  )]<\infty$. Moreover, because the inter-arrivals
have unbounded support a geometric trial argument will yield that if
$\delta>0$ is chosen sufficiently small then $E[\exp\left(  \delta\tau\left(
0\right)  \right)  ]<\infty$. In turn, this bound implies that $E[\exp\left(
\delta\bar{\tau}\right)]  <\infty$.

Next we want to show that the output indeed follows the target steady state
distribution. This portion follows precisely from the validity of the DCFTP protocol.

Using the similar notation of $Y^{+}(t;T,y)$, we define $Y(t;T,y)$ as the
number of customers in a GJN start with $Y(0;T,y)=y$ and is driven by the same
sequence of inter-arrival times, service requirements and routing indices as
$\mathcal{N}^{+}$ on $[-T,0]$. Given the comparison results in Theorem
\ref{Thm_Domination}, given that $Y^{+}(-T^{\prime\prime})=0$, we can conclude
that for all $T>T^{\prime\prime}$,
\[
\sum_{i}Y_{i}(T-T^{\prime\prime};T,0)\leq\sum_{i}Y_{i}(T-T^{\prime\prime
+}(-T))\leq\sum_{i}Y_{i}^{+}(T-T^{\prime\prime+}(-T))=\sum_{i}Y^{+}%
_{i}(-T^{\prime\prime})=0
\]
and hence $Y(T-T^{\prime\prime};T,0)=0$. Therefore, for any $T>T^{\prime
\prime}$
\[
Y(T;T,0)=Y(T; T, Y(T-T^{\prime\prime};T,0))=Y(T^{\prime\prime};T^{\prime
\prime},0).
\]
As the process $Y(\cdot)$ has a unique stationary distribution (see
\cite{Sigman}), we can conclude $Y(T^{\prime\prime};T^{\prime\prime}%
,0)=\lim_{T\to\infty}Y(T;T,0)$ follows the stationary distribution.

\textbf{Acknowledgement:} Blanchet acknowledges support from the NSF through
the grants CMMI-0846816 and 1069064. Chen acknowledges support from the NSF through the grant CMMI-1538102.

\end{document}